\let\ORIlabel\label
\let\ORIrefstepcounter\refstepcounter
\AddToHook{package/hyperref/before}{
  \let\label\ORIlabel
  \let\refstepcounter\ORIrefstepcounter
}

\documentclass[hidelinks,onefignum,onetabnum]{siamart220329}

\usepackage{lipsum}
\usepackage{amsfonts}
\usepackage{graphicx}
\usepackage{epstopdf}
\usepackage{algorithmic}
\usepackage{hyperref}
\hypersetup{
colorlinks=true,
allcolors = green!50!black,
urlcolor=red!50!black,
}
\usepackage{mathrsfs}
\usepackage{geometry}
\usepackage{multirow}
\usepackage{graphics,graphicx,epsf,subfigure,epstopdf,epsfig}
\usepackage{amssymb}
\usepackage{amsmath}
\usepackage{float}
\usepackage{boxedminipage}
\usepackage{caption}
\usepackage{enumerate}
\usepackage{enumitem}
\usepackage{accents}
\usepackage{fancybox}
\usepackage{bm}
\usepackage{algorithm}
\usepackage{algorithmic}
\usepackage{listings}
\usepackage{booktabs}
\usepackage{url}
\usepackage{dcolumn}  % 引入 dcolumn 宏包

\usepackage{colortbl}  %彩色表格需要加载的宏包
\usepackage{xcolor}
\usepackage{array}   %对表列和表格线的设置需要用到array宏包

\usepackage{arydshln}  % 支持虚线表格线
\usepackage{booktabs}  % 支持 \cmidrule
\ifpdf
  \DeclareGraphicsExtensions{.eps,.pdf,.png,.jpg}
\else
  \DeclareGraphicsExtensions{.eps}
\fi

\newsiamremark{remark}{Remark}
\newsiamremark{hypothesis}{Hypothesis}
\crefname{hypothesis}{Hypothesis}{Hypotheses}
\newsiamthm{claim}{Claim}

\headers{Smoothing Anderson acceleration algorithm}{Zekai Li and Wei Bian}

\title{A smoothing Anderson acceleration algorithm for nonsmooth fixed point problems with linear convergence\thanks{Submitted to the editors {December 6, 2024}.
\funding{This work was funded by the National Natural Science Foundation of China (project numbers: 12425115, 12271127), Joint Guidance Project of Heilongjiang Provincial Natural Science Foundation (project number: GNCHSF20260040) and the Fundamental Research Funds for the Central Universities (project number: HIT.DZJJ.2026069).}}}
\author{Zekai Li\thanks{Department of Mathematics, Harbin Institute of Technology, Harbin, China (\email{zeeklc\_op@163.com}).}
\and Wei Bian\thanks{Corresponding author. Department of Mathematics, Harbin Institute of Technology, Harbin, China (\email{bianweilvse520@163.com}).}}

\usepackage{amsopn}
\makeatletter
\newcommand*{\addFileDependency}[1]{% argument=file name and extension
  \typeout{(#1)}% latexmk will find this if $recorder=0 (however, in that case, it will ignore #1 if it is a .aux or .pdf file etc and it exists! if it doesn't exist, it will appear in the list of dependents regardless)
  \@addtofilelist{#1}% if you want it to appear in \listfiles, not really necessary and latexmk doesn't use this
  \IfFileExists{#1}{}{\typeout{No file #1.}}% latexmk will find this message if #1 doesn't exist (yet)
}
\makeatother

\ifpdf
\hypersetup{
  pdftitle={A smoothing Anderson acceleration algorithm for nonsmooth fixed point problems with linear convergence},
  pdfauthor={Zekai Li and Wei Bian}
}
\fi

\renewcommand{\u}{\bm{u}}
\newcommand{\x}{\bm{x}}
\newcommand{\A}{\bm{A}}
\newcommand{\B}{\bm{B}}
\newcommand{\I}{\bm{I}}
\renewcommand{\b}{\bm{b}}
\renewcommand{\v}{\bm{v}}

\newcommand{\z}{\bm{z}}

\newcommand{\p}{\bm{p}}

\renewtheorem{remark}{\bf Remark}[section]
\newtheorem{example}{\bf Example}[section]
\newtheorem{assumption}{Assumption}[section]

\begin{document}

\maketitle

% REQUIRED
\begin{abstract}
    In this paper, we consider the Anderson acceleration method for solving the contractive fixed point problem, which is nonsmooth in general. We define a class of smoothing functions for the original nonsmooth fixed point mapping, which are primarily applicable to problems involving max-type nonsmoothness (see \cref{section3}). Based on the Anderson acceleration technique, we propose the Smoothing Anderson(m) algorithm, in which we utilize a smoothing function of the original nonsmooth fixed point mapping and update the smoothing parameter adaptively. We first demonstrate the r-linear convergence of the proposed Smoothing Anderson(m) algorithm for solving the considered nonsmooth contractive fixed point problem with r-factor no larger than $c$, where $c$ is the contractive factor of the fixed point mapping. Second, we establish that both the Smoothing Anderson(1) and the Smoothing EDIIS(1) algorithms are q-linearly convergent with q-factor no larger than $c$. Finally, we present three numerical examples with practical applications arising from elastic net regression, free boundary problems for infinite journal bearings and non-negative logistic regression to illustrate the superior performance of the proposed Smoothing Anderson(m) algorithm compared with some popular methods.
\end{abstract}

% REQUIRED
\begin{keywords}
    nonsmooth fixed point problem, Anderson acceleration, smoothing method, r-linear convergence, q-linear convergence.
\end{keywords}
  % REQUIRED  
\begin{AMS}
65H10, 49J52, 68W25
\end{AMS}

\section{Introduction}
In this paper, we consider the following fixed point problem 
\begin{equation}\label{eq1-0}
\u=G(\u),
\end{equation}
where $G: \mathbb{R}^n\rightarrow \mathbb{R}^n$ is Lipschitz continuous but not necessarily smooth. 
We assume that $G$ has fixed points in a closed set $D\subset\mathbb{R}^n$ and $G$  is a contraction mapping in $D$ with factor $c\in(0,1)$ in the Euclidean norm $\|\cdot\|$, i.e., 
\begin{equation}\label{c}
\|G(\u)-G(\v)\|\le c \| \u-\v \|, \quad \forall \u,\v \in D.
\end{equation}
By the contraction mapping theorem \cite{Ortega}, the function $G$ has a unique fixed point $\u^*\in D$, which is the unique solution of the nonlinear equations
\[
F(\u):=G(\u)-\u=0.
\]
We denote the residual of \cref{eq1-0} at $\u$ by $F(\u)$ and it satisfies
\begin{equation}\label{eqF1}
(1-c)\|\u-\u^*\|\leq \|F(\u)\| \leq (1+c)\|\u-\u^*\|,\quad \forall \u\in D.
\end{equation}

Nonsmooth fixed point problems exist in many fields and have a wide range of applications in natural science and social science. First, many optimization problems admit a fixed point representation of their optimality conditions, and numerous iterative algorithms are essentially procedures for finding such fixed points. A canonical example is the constrained smooth minimization problem $\min_{\x\in\Omega}f(\x)$, which is commonly reformulated via optimality conditions as the problem of finding a fixed point of $G(\u)=P_{\Omega}(\u-\lambda\nabla f(\u))$, that is nonsmooth in general due to the projection operator $P_{\Omega}$. Furthermore, proximal gradient algorithms, developed for denoising and sparse recovery models like LASSO \cite{tibshirani1996regression}, also admit fixed point frameworks. As evidenced by the widely used Iterative Shrinkage Thresholding Algorithm (ISTA) \cite{daubechies2004iterative}, most of these corresponding mappings are inherently nonsmooth. Second, numerous practical problems can be transformed into fixed point problems via suitable transformations, and the resulting mappings are often inherently nonsmooth. For example, the nonlinear complementarity problem (NCP) \cite{more1996global} can be formulated as a variational inequality, and can then be equivalently reformulated as a nonsmooth fixed point problem via the $\max$ operator. Moreover, in machine learning, iterative inference and deep equilibrium models \cite{bai2019deep} can be interpreted as finding the fixed point of a neural network-parameterized nonlinear transformation, which is, however, often nonsmooth in practice. The Anderson acceleration method is an essential technique for solving fixed point problems \cref{eq1-0}. D. Anderson \cite{Anderson} first proposed this method for resolving integral equations in 1965. 
Since Anderson acceleration only requires the function values (zeroth-order information) of the fixed point mapping $G$, rather than the first-order information (gradient) of $G$, it performs efficiently in numerical calculations for practical applications. If $G$ is defined by the first-order optimality condition or derived from a gradient descent scheme of an optimization problem, then the Anderson acceleration method corresponds to a first-order algorithm. As a result, it is widely used in electronic structure computations \cite{Anderson,ChenKelley2015,TothKelley2015,Walker}, geometry optimization problems \cite{peng2018anderson} and machine learning \cite{wei2021class}. 

\begin{algorithm*}[t]
    \caption{Anderson(m)}\label{algo-o}
    Choose $\u_0\in D$ and a positive integer $m$. Set $\u_1=G(\u_0)$ and $F_0=G(\u_0)-\u_0$. \\
    \textbf{for $k=1,2,...$ do}
    
    \qquad  set $F_k=G(\u_k)-\u_k$;
    
    \qquad choose $m_k=\min\{m,k\}$;
    
    \qquad solve
    \begin{equation}\label{Anderson_alpha}
    \min\,\,\left \|\sum_{j=0}^{m_k}\alpha_j{F}_{k-m_k+j}\right \|\quad
    \mbox{s.t.}\,\sum_{j=0}^{m_k}\alpha_j=1
    \end{equation}
    \qquad to find a solution $\{\alpha_j^k:j=0,\ldots,m_k\}$, and set
    \[
    \u_{k+1}=\sum_{j=0}^{m_k}\alpha_j^k{G}(\u_{k-m_k+j});\]
    \textbf{end for}
\end{algorithm*}

The core of Anderson acceleration is to use the information of historical iteration points and update the next iterate by a combination of the previous iterates. At iteration $k$, with the maximum allowable algorithmic depth parameter $m$ and the algorithmic depth parameter $m_k$, Anderson(m) algorithm is described as in \cref{algo-o}, which stores the function values of $G(\u)$ and $F(\u)$ at $\u_{k-m_k+j}$, $j=0,\ldots, m_k$. We can get a set of combination coefficients by solving a linear constrained quadratic convex optimization problem modeled in \cref{Anderson_alpha}. Using the obtained optimal combination coefficients, Anderson(m) algorithm aims to define the new iterate $\u_{k+1}$ by the linear combination of the function values at the last $m_k+1$ iterations. Obviously, Anderson(0) is the Picard method, i.e.,
\begin{equation}\label{Picard}
\u_{k+1}=G(\u_k).
\end{equation}
If the initial point $\u_0$ is chosen from ${\cal B}(\rho):=\{\u: \|\u-\u^*\|\leq \rho\}\subset D$ with $\rho>0$, then the Picard iteration is q-linearly convergent, that is,
\begin{equation}\label{q-linear}
\|\u_{k+1}-\u^*\|\le c\|\u_k-\u^*\|, \quad \forall k\geq 0.
\end{equation}
Some researchers proposed variations of Anderson(m) algorithm by modifying the constraints in \cref{algo-o}. For example, by adding a non-negative constraint in \cref{Anderson_alpha}, the Anderson(m) algorithm reduces to EDIIS(m) algorithm \cite{ChenKelley2015, garza2012comparison, kudin2002black}, for which  \cref{Anderson_alpha} is replaced by a new minimization problem as follows: 
\begin{equation}\label{EDIIS_alpha}
\min\,\,\left \|\sum_{j=0}^{m_k}\alpha_j{F}_{k-m_k+j}\right \|\quad
\mbox{s.t.} \,\sum_{j=0}^{m_k}\alpha_j=1, \,  \alpha_j\ge 0,\quad  j=0,\ldots, m_k.
\end{equation}

Although Anderson acceleration has achieved great success due to its efficient numerical performance, there are still few results on its convergence analysis. The first convergence results for Anderson(m) algorithm were proposed in 2015 by Toth and Kelley \cite{TothKelley2015} for both linear and nonlinear problems. They proved the r-linear convergence of the Anderson(m) algorithm with r-factor $\hat{c}\in (c,1)$ when $G$ is Lipschitz continuously differentiable and the linear combination coefficients remain bounded. Additionally, they established the r-linear convergence of Anderson(1) algorithm without such condition on the coefficients. On this basis, Chen and Kelley \cite{ChenKelley2015} demonstrated the global r-linear convergence of the EDIIS(m) algorithm without the differentiability of $G$, and its r-factor is $\hat{c}=c^{1/(m+1)}$. 
Wei et al. \cite{wei2021class} proposed a novel short-term recurrence variant of Anderson acceleration to reduce the memory overhead. The work \cite{wei2021class} also analyzed the convergence properties and demonstrated the local r-linear convergence under the assumption that $G$ is Lipschitz continuously differentiable. Furthermore, for a class of integral equations, in which the operator can be decomposed into smooth and nonsmooth parts in a Hilbert space, Bian et al. \cite{BCK} showed the convergence results of Anderson(m) and EDIIS(m) algorithms, not only the r-linear convergence with r-factor $\hat{c}\in (c,1)$ when $m\geq 1$, but also including the q-linear convergence with q-factor $c$ when $m=1$.
Evans et al. \cite{evans2020proof} proposed a one-step analysis of Anderson acceleration for contractive mappings, demonstrating that Anderson acceleration can actually improve the convergence rate with the introduction of higher order terms. Moreover, Pollock and Rebholz \cite{pollock2021anderson} further extended the convergence results to some noncontractive mappings, which provided insight into the balance between the contributions of higher and lower order terms.

Considering the classical Anderson acceleration may suffer from stagnation and instability, various techniques have been developed to address these challenges. Zhang et al. \cite{zhang2020globally} studied a new variant of Anderson acceleration with safeguarding steps and a restart checking strategy for the general fixed point problem which is potentially nonsmooth. Once the iteration point satisfies certain conditions, then the algorithm accepts an Anderson acceleration step. Although they proved the global convergence of residual for this algorithm, they did not obtain the relevant convergence rate. Another approach is to introduce regularization terms in \cref{Anderson_alpha}. Ouyang et al. \cite{ouyang2023nonmonotone} adopted a nonmonotone trust-region framework with an adaptive quadratic regularization, and designed a globalization scheme for Anderson acceleration by implementing a reasonable acceptance mechanism. They proved that the algorithm is globally convergent in terms of residuals for a class of nonexpansive mappings and exhibits local r-linear convergence for contractive mappings. In \cite{fu2020anderson}, Fu et al. developed an Anderson accelerated Douglas-Rachford splitting (A2DR) algorithm by incorporating both safeguarding steps and regularization terms for the nonsmooth convex optimization problem with linear constraints, where the objective function is only accessible through its proximal operator. It is further shown that the algorithm either converges globally or identifies infeasibility under mild conditions. 
So far, although some convergence results have been established, these studies either fail to demonstrate the linear convergence factor, or only obtain local convergence, or require certain smoothness conditions to guarantee convergence. 
These analyses primarily focus on the convergence behavior of Anderson acceleration, but there are still many theoretical results from other perspectives. In \cite{mai2020anderson}, Mai and Johansson introduced a novel perspective for adapting Anderson acceleration to proximal gradient algorithms. On the one hand, they extended the original local convergence results to the proposed nonsmooth setting, where the fixed point mapping is required to be differentiable in a neighborhood of the fixed point. On the other hand, they showed by a counterexample that the vanilla Anderson acceleration is not globally convergent, and then proposed a globalized variant by introducing the safeguarding steps in the context of proximal gradient descent schemes. Ouyang et al. \cite{ouyang2024descent} investigated the local properties of Anderson acceleration with restarting (AA-R) in terms of function values. They demonstrated that {AA-R} is a local descent method, achieving a more significant reduction in the objective function values compared to the basic gradient methods. Also, they proposed a global strategy for {AA-R}. Additionally, there are numerous significant results on damping iterations, in which the new iterate is written in the form of a combination of $\u_{k-m_k+j}$ and ${G}(\u_{k-m_k+j})$. Evans et al. \cite{evans2020proof} focused on the effect of damping parameters on Anderson acceleration, examining both fixed and adaptive damping parameters. Their numerical experiments demonstrated that the adaptive damping parameter can enhance convergence stability and robustness for noncontractive mappings with the variation of $m$. Motivated by this work, Chen and Vuik \cite{chen2024non} extended the research by proposing a new variant of Anderson acceleration that minimizes the residual at each iteration through optimized damping. Their experiments illustrated that Anderson acceleration with optimized damping procedure often converges significantly faster than Anderson acceleration with fixed damping, adaptive damping or without damping, especially when the depth $m$ is large. Although these two studies indicated that selecting appropriate damping parameters can improve convergence, there remains a lack of consistent theoretical results supporting these empirical findings. All of these pose challenges for the research of Anderson acceleration. In this paper, we focus on solving the nonsmooth contractive fixed point equation \eqref{eq1-0} by Anderson(m) algorithm and analyzing its convergence results.

Smoothing approximation methods have been studied for decades and play an important role in numerous applications. Smoothing methods are primarily characterized by their capability to approximate nonsmooth functions using parameterized smooth functions, with the flexibility to update the smoothing parameters as the algorithm iterates. As a result, we can approximate the original nonsmooth problem by many smooth problems with different smoothing parameters, so it can be transformed into a series of smooth problems, for which there exist many efficient methods. 
With the smoothing approximation method, Bian and Chen \cite{bian2022anderson} proposed a new smoothing Anderson acceleration for the contractive fixed point problem composited with `max' operator. They proved its local r-linear convergence with factor $c$ for the problem, which is the same as the convergence rate of Anderson acceleration for the continuously differentiable case. Additionally, they showed that Anderson(1) algorithm is q-linearly convergent with a new q-factor, which is strictly smaller than the q-factor given in \cite{BCK, TothKelley2015}.

The contributions of this paper are new convergence results on Anderson(m) and EDIIS(m) algorithms for the nonsmooth contractive fixed point problem \cref{eq1-0} by using the smoothing function of mapping $G$, which further extends the class of nonsmooth functions compared with \cite{bian2022anderson,BCK,fu2020anderson,mai2020anderson}. The proposed method in this paper not only achieves a potentially better convergence factor than that in \cite{BCK}, but also further weakens the differentiability condition in \cite{TothKelley2015}.
Compared to the results in \cite{bian2022anderson}, our study investigates a broader class of nonsmooth contractive mappings, instead of being restricted to the particular case composited with `max' operator. Moreover, the proposed Smoothing Anderson(m) algorithm with an adaptive updating scheme for $\mu$ overcomes the limitation that the algorithm in \cite{bian2022anderson} relies heavily on a fixed smoothing parameter, which essentially serves as a hyperparameter and requires careful tuning. Besides, the convergence factor of the algorithm with a fixed smoothing parameter may be strictly greater than ours (see \cref{remark24}). Although \cite{mai2020anderson} also concentrates on nonsmooth fixed point problems, it still requires the fixed point mapping to be differentiable at the fixed point, which is fundamentally different from ours.
The work \cite{fu2020anderson} aims to propose a variant of Anderson acceleration specifically tailored to the Douglas-Rachford splitting algorithm. In contrast, we focus more on establishing the linear convergence results for nonsmooth contractive fixed point problems.

The remainder of this paper is organized as follows. In \cref{section2}, inspired by the convergence analysis of Anderson(m) algorithm for a continuously differentiable function $G$ and the smoothing approximation method, we introduce a Smoothing Anderson(m) algorithm and demonstrate its local r-linear convergence for problem \cref{eq1-0} with factor ${c}$, which is the same as the contraction factor of $G$.
Here, we would like to stress that we introduce some conditions on the used smoothing function for $G$ to guarantee this convergence result, where all the conditions are put in the definition of smoothing function for easy reading. Subsequently, we establish that both Smoothing Anderson(1) and Smoothing EDIIS(1) algorithms are q-linearly convergent for problem \cref{eq1-0} with factor ${c}$. In \cref{section3}, we present three actual applications and give the corresponding numerical examples including elastic net regression, free boundary problems for infinite journal bearings, and non-negative logistic regression problem, to illustrate our theoretical results. Additionally, it is of significance to explain how these problems can be reformulated as nonsmooth fixed point problems and how to appropriately construct a smoothing function. Preliminary numerical results show that the proposed Smoothing Anderson(m) algorithm has more efficient numerical performance than other methods for the nonsmooth fixed point problems.

\section{Smoothing Anderson(m) and its local convergence}\label{section2}
In this section, we propose a variant of Anderson(m) that incorporates the smoothing method, referred to as Smoothing Anderson(m) and presented in Algorithm 2.1. In subsection \ref{section2.2}, we prove that Smoothing Anderson(m) converges r-linearly to the solution of \cref{eq1-0} with factor $c$, which matches the contraction factor of $G$. In subsection \ref{section2.3}, we provide an analysis of the q-linear convergence of Smoothing Anderson(1). To support our analysis, we rely on a smoothing function ${\mathcal{G}}(\u,\mu)$ of $G$ as defined in \cref{ass3}. Its feasibility is illustrated through examples and practical problems in section \ref{section3}.

\begin{algorithm}[t]
    \caption{Smoothing Anderson(m)}\label{algo-s}
    Choose $\u_0\in D$ and set $\mu_0=\sqrt{\|F(\u_0)\|}$. 
    
    Set $k=0$.

    % \textbf{for $k=0,1,2,...$ do}

    \textbf{while $\|F(\u_k)\|\neq 0$}
    
    \qquad choose $m_k=\min\{m,k\}$;
    
    \qquad calculate $\mathcal{G}_{k} = \mathcal{G}(\u_k, \mu_k)$;
    
    \qquad set ${\mathcal{F}}_{k}= {\mathcal{G}}_{k}-\u_{k};$
    
    \qquad solve $\{\alpha_j^k:j=0,\ldots,m_k\}$ by
    \begin{equation}\label{s-anderson-m}
    \min\,\,\left\|\sum_{j=0}^{m_k}\alpha_j^k{\mathcal{F}}_{k-m_k+j}\right\|\quad
    \mbox{s.t.}\,\sum_{j=0}^{m_k}\alpha_j^k=1;
    \end{equation}
    \qquad set
    \begin{equation}\label{eq-u-s}
    \u_{k+1}=\sum_{j=0}^{m_k}\alpha_j^k{\mathcal{G}}_{k-m_k+j},\,\,\mu_{k+1}=\frac{1}{\sqrt{\|F(\u_0)\|}}\max_{0\leq j\leq m_k}\|{\mathcal{F}}_{k-m_k+j}\|;
    \end{equation}
    \qquad set $k=k+1$.

    \textbf{end while}
    
    % \textbf{end for}
\end{algorithm}

\begin{remark}\label{remark1}
    When $k=0$ and $m_k = 0$, $\{\alpha_j^k:j=0,\ldots,m_k\}$ is a singleton $\{1\}$, which means that \cref{eq-u-s} is equivalent to a similar Picard iteration for $k = 0$, i.e., $\u_1=\mathcal{G}_0$ and $\mu_1=\frac{1}{\sqrt{\|F(\u_0)\|}}\|{\mathcal{F}}_{0}\|$.
\end{remark}

\begin{definition}\label{ass3}
We call the function ${\mathcal{G}}:\mathbb{R}^n\times(0,\tilde{\mu}]\rightarrow\mathbb{R}^n$ with $\tilde{\mu}>0$ a smoothing function of $G$ which is Lipschitz continuous on ${\cal B}(\rho)\subset D$ with $\rho>0$, if it is Lipschitz continuously differentiable with respect to $\u$ for any fixed $\mu\in(0,\tilde{\mu}]$, and there exist $0<\delta\leq\rho$ and $\hat{\mu}\in(0,\tilde{\mu}]$ such that it has the following properties.
\begin{itemize}
\item [{\rm (i)}] There exists $\kappa>0$ such that
\[\|{\mathcal{G}}(\u,\mu)-G(\u)\|\leq\kappa\mu,\quad \forall \u \in\mathcal{B}(\delta),\;\mu\in(0,\hat{\mu}].\]
\item [{\rm (ii)}] There exists a bounded continuous function $\tau$ defined on $[0,\hat{\mu}]$ with $\tau(0)={c}$ and $\tau({\mu})\leq\bar{c}<1$, $\forall\mu\in[0,\hat{\mu}]$ such that ${\mathcal{G}}(\u,\mu)$ is a contractive mapping on $\mathcal{{B}}(\delta)$
    with factor $\tau(\mu)$ for any fixed $\mu\in(0,\hat{\mu}]$.
\item [{\rm (iii)}]
   ${\mathcal{G}}(\u^*,\mu)=G(\u^*)$ and ${\mathcal{G}}'(\u^*,\mu)={\mathcal{G}}'(\u^*,\nu)$, $\forall \mu,{\nu}\in(0,\hat{\mu}]$.
\item [{\rm (iv)}] There exists $\beta>0$ such that it holds
\[\mathcal{G}(\u,\mu)=\mathcal{G}(\u^*,\mu)+\mathcal{G}'(\u^*,{\mu})(\u-\u^*)+\Delta_{\u,\mu},\quad \forall \u \in \mathcal{B}(\delta),\, \mu\in(0,\hat{\mu}],\]
 where $\|\Delta_{\u,\mu}\| \leq{\beta\|\u-\u^*\|^2}/{2\mu}$.
\end{itemize}
\end{definition}

\begin{remark}
 Let $f:\mathbb{R}^n\rightarrow\mathbb{R}$ be contractive, and let $\psi: \mathbb{R}^n \rightarrow [0, \infty)$ be a continuously differentiable probability density function with compact support $D\subset\mathbb{R}^n$ containing the origin, and satisfying $\int_D\psi(\z){\rm d}\z = 1$. We then can construct a smoothing function $\tilde{f}: \mathbb{R}^n \times[0, \tilde{\mu}]\rightarrow\mathbb{R}$ via convolution \cite{rockafellar1998variational} as
        \begin{equation*}\label{smoothing}
        \tilde{f}(\x, \mu) = \int_{\mathbb{R}^n}f(\x-\z)\mu^{-n}\psi(\z/\mu) {\rm d}\z=\int_{\mathbb{R}^n}f(\x-\mu \z)\psi(\z){\rm d}\z. 
        \end{equation*}
    It is easy to verify that the smoothing function $\tilde{f}$ satisfies \cref{ass3}-(i),(ii) and (iv). As for condition (iii), although it may appear difficult to satisfy, it is essential that we need to identify the points at which the nonsmooth function fails to be differentiable.  Moreover, if $\tilde{f}$ is a smoothing function of $f$ defined as in Definition 2.1, then $\mathcal{G}(\u,\mu):=H(\tilde{f}(Q(u),\mu))$ is a smoothing function of $G(u):=H(f(Q(u)))$ when $H$ and $Q$ are Lipschitz continuously differentiable. 
\end{remark}

The contraction factor of $\mathcal{G}$ depends on the structure of the proposed smoothing function. 
For a Lipschitz continuously differentiable and contractive function $Q : \mathbb{R}^n \rightarrow \mathbb{R}^n$, 
we can give a smoothing function of $G(\u) = \max\{Q(\u), 0\}$ as an example to satisfy Definition \ref{ass3}.
\begin{example}\label{ex1}
A smoothing function of $\max\{t, 0\}$ can be defined as follows: 
for any $t\in\mathbb{R}$ and $\mu\in(0,\tilde{\mu}]$, let
\begin{equation}\label{sm1}
    \phi(t,\mu)=\left\{
    \begin{aligned}
    &0&&\mbox{if $t<0$}\\
    &{t^2}/{2\mu}&&\mbox{if $0\leq t\leq\mu$}\\
    &(t-\mu)^2/4+t-\mu/2&&\mbox{if $\mu<t\leq\mu+\sqrt{\mu}$}\\
    &-(t-\mu-2\sqrt{\mu})^2/4+t&&\mbox{if $\mu+\sqrt{\mu}\leq t\leq\mu+2\sqrt{\mu}$}\\
    &t&&\mbox{if $t>\mu+2\sqrt{\mu}$.}\\
    \end{aligned}
    \right.
\end{equation}
Using $\phi$ in (\ref{sm1}), we set $\mathcal{G}(\u,\mu)=\Phi(Q(\u),\mu)$, 
where $\Phi(\v,\mu)=(\phi(v_1,\mu), \ldots,\phi(v_n,\mu))^T$ for $\v=(v_1, \ldots, v_n)^T\in\mathbb{R}^n$. In \cite{bian2022anderson}, it has been proved that $\mathcal{G}(\u, \mu)$ is a smoothing function of $G(\u):=\max\{Q(\u),0\}$ satisfying the conditions in \cref{ass3} with $\tilde{\mu} < +\infty$ and $\hat{\mu} = \min\{\tilde{\mu}, 3(1-c)^2/c^2, (\varrho/3)^2, 1\}$, where $\varrho=\min\{Q_i(\u^*): Q_i(\u^*)>0, i=1,2, \dots, n\}$. 

Though the mapping 
$G(\u)=\max\{Q(\u), 0\}$ defined above is strongly semismooth (see \cite[Proposition 2.16]{ulbrich2011semismooth}), it is not necessary for the original function to be strongly semismooth in order to possess a smoothing function satisfying \cref{ass3}. For example: consider $\u\in\mathbb{R}$ and define
\begin{equation*}
G(\u) = \left\{ 
            \begin{aligned}
            &\frac{1}{2}\u^{\frac{3}{2}}&&\mbox{if $\u\geq0$}\\
            & -\frac{1}{2}\u &&\mbox{if $\u<0$}\\
            \end{aligned}\right.
\quad
\mbox{and}
\quad
\mathcal{G}(\u,\mu) = \left\{ 
                        \begin{aligned}
                            &\frac{1}{2}\phi(\u,\mu)^{\frac{3}{2}} &&\mbox{if $\u\geq0$}\\
                            & \frac{1}{2}\phi(-\u,\mu) &&\mbox{if $\u<0$,}\\
                        \end{aligned}\right.
\end{equation*}
where $\phi$ is the same as in \cref{sm1}. Although the above $G$ is not strongly semismooth, it admits a smoothing function satisfying \cref{ass3}.
\end{example}

Throughout this paper, we denote $\mathcal{G}$ a smoothing function of $G$ defined in \cref{ass3} and
denote $\mathcal{F}$ the smoothing residual function given by
$$\mathcal{F}(\u,\mu)=\mathcal{G}(\u,\mu)-\u.$$
In what follows, when it is clear from the context, the Jacobian of $\mathcal{G}(\u,\mu)$ with respect to $\u$ is simply denoted by $\mathcal{G}'(\u,\mu)$.
For the sake of simplicity, we denote $\mathcal{G}(\u_k,\mu_k)$ and $\mathcal{F}(\u_k,\mu_k)$ by $\mathcal{G}_k$ and $\mathcal{F}_k$, respectively. Similar to Anderson(m), we only use the function values of ${\mathcal{G}}(\u,\mu)$ at the current iteration. Different from it, we need to update $\mu_k$ at each iteration. 

Moreover, we notice that the exact values of $\hat{\mu}$, $\delta$, $\kappa$ and $\beta$ in \cref{ass3} are only used in the analysis, but not needed in the implementation of Smoothing Anderson(m).

From \cref{ass3}-(ii) and (iii), we have
\begin{equation}\label{eqFS1}
(1-\tau({\mu}))\|\u-\u^*\|\leq\|\mathcal{F}(\u,\mu)\|\leq(1+\tau({\mu}))\|\u-\u^*\|,\quad\forall \u\in \mathcal{B}(\delta),\,\mu\in(0,\hat{\mu}].
\end{equation}

\subsection{r-linear convergence of Smoothing Anderson(m) with \texorpdfstring{$\mathbf{m \geq 1}$}{}}\label{section2.2}

Let us mention that the optimality condition of $\{\alpha_j^k\}$ for optimization problem \cref{s-anderson-m} is not used in the proof. Indeed, we only need the following conditions on them, which are the same as the conditions in \cite{ChenKelley2015,TothKelley2015}.
\begin{assumption}\label{ass2}$ $
\begin{itemize}
\item [{\rm (i)}] $\|\sum\nolimits_{j=0}^{m_k}\alpha_j^k{\mathcal{F}}_{k-m_k+j}\|\leq\|{\mathcal{F}}_k\|$;
\item [{\rm (ii)}] $\sum\nolimits_{j=0}^{m_k}\alpha_j^k=1$;
\item [{\rm (iii)}] there exists a constant $M_{\alpha}\geq 1$ such that $\sum_{j=0}^{m_k}|\alpha_j^k|\leq M_{\alpha}$ holds for all $k\geq1$.
\end{itemize}
\end{assumption}

To prepare the proof of the  main result and to prove r-linear convergence of the residual function values $\|F(\u_k)\|$, we now present some preliminary results.
\begin{lemma}\label{lemma-uk}
Suppose \cref{ass2} holds and
\begin{equation}\label{eq3}
\|{\mathcal{F}}_k\|\leq\|{\mathcal{F}}_0\|,\quad \forall 0\leq k\leq K.
\end{equation}
For the $\delta$ and $\hat{\mu}$ in \cref{ass3}, and any ${\mu}\in(0,\hat{\mu}]$, if $\u_0$ is sufficiently close to $\u^*$, then the following statements hold.
\begin{itemize}
\item [{\rm (i)}] $\mu_k\in(0,{\mu}]$, $\forall 0\leq k\leq K+1$.
\item [{\rm (ii)}] $\u_{k}\in\mathcal{B}(\delta)$, $\forall 0\leq k\leq K+1$.
\item [{\rm (iii)}] $\sum_{j=0}^{m_K}\alpha_j^K\u_{K-m_K+j}\in\mathcal{B}(\delta)$.
\end{itemize}
\end{lemma}
\begin{proof}
Firstly, set
\begin{equation}\label{u0}
\|\u_0-\u^*\|\leq\delta_0:=\min\left\{
\frac{(1-c)^2{\mu}^2}{4(1+c)},\frac{(1-\bar{c})\delta}{2M_{\alpha}}\right\}.
\end{equation}
By \cref{ass3} and ${\mu}\leq\hat{\mu}$, if $\mu_k\in(0,{\mu}]$, ${\mathcal{G}}(\u,\mu_k)$ is a contractive mapping on $\mathcal{B}(\delta)$ with fixed point $\u^*$ and factor $c_k:=\tau({\mu_k})$.
We will prove that if $\u_0\in\mathcal{B}(\delta_0)$, then the results in (i)-(iii) hold.

(i) For $k=0$, by $\mu_0=\sqrt{\|F(\u_0)\|}\leq\sqrt{(1+c)\|\u_0-\u^*\|}$, we can get $\mu_0\in(0,{\mu}]$.
Combining \cref{eqF1} with \cref{eqFS1}, we have
\begin{equation}\label{eq-FS}
    \|\mathcal{F}(\u,\mu)\| \leq(1+\tau({\mu}))\|\u-\u^*\| \leq\frac{1+\tau({\mu})}{1-c}\|F(\u)\|,\quad\forall\mu\in(0,\hat{\mu}],\,\u\in\mathcal{B}(\delta).
\end{equation}
Then, for $k = 0, \ldots, K$, combining the results of (2.5) and (2.7) yields
\begin{equation}\label{28}
\begin{aligned}
    \mu_{k+1}=&\frac{1}{\sqrt{\|F(\u_0)\|}}\max_{0\leq j\leq m_k}\|{\mathcal{F}}_{k-m_k+j}\|
    \leq\frac{\|\mathcal{F}_0\|}{\sqrt{\|F(\u_0)\|}}
    \leq\frac{(1+\tau({\mu}_0))\|F(\u_0)\|}{(1-c)\sqrt{\|F(\u_0)\|}}\\
    \leq&\frac{2}{1-c}\sqrt{\|F(\u_0)\|}
    \leq\frac{2}{1-c}\sqrt{(1+c)\|\u_0-\u^*\|},
\end{aligned}
\end{equation}
from \cref{u0}, which implies $\mu_{k}\in(0,{\mu}]$ for all $k=1,\ldots,K+1$. In conclusion, it can be derived that $\mu_k\in(0,{\mu}]$ for all $k=0,\ldots,K+1$.

(ii) We prove the result by induction on $K$. Since $\u_0\in \mathcal
{B}(\delta)$ by \cref{u0}, suppose that $\u_k\in\mathcal{B}(\delta)$ for all $0\leq k\leq K$. For $0\leq k\leq K$, since $\mu_k\in(0,{\mu}]$, and $\tau(\cdot)$ is bounded on $(0,\mu]$, we have $c_k\leq\bar{c}$.
Then, combining \cref{eqFS1} with \cref{eq3}, we have
\begin{equation}\label{eq11}
\|\u_k-\u^*\| \leq\frac{\|{\mathcal{F}}_k\|}{1-c_k}\leq
\frac{\|{\mathcal{F}}_0\|}{1-c_k}
\leq\frac{(1+c_0)\|\u_0-\u^*\|}{1-c_k}
\leq\frac{2\|\u_0-\u^*\|}{1-\bar{c}}.
\end{equation}
% which implies that $\u_k\in\mathcal{B}(\delta)$ for all $0\leq k\leq K$ by $\|\u_0-\u^*\| \leq\frac{1-\bar{c}}{2}\delta$.
Recalling (\ref{eq-u-s}), we have
$$\begin{aligned}
\|\u_{K+1}-\u^*\|=&\left \|\sum_{j=0}^{m_K}\alpha_j^K{\mathcal{G}}_{K-m_K+j}-\u^* \right \|
=\left \|\sum_{j=0}^{m_K}\alpha_j^K({\mathcal{G}}_{K-m_K+j}-{\mathcal{G}}(\u^*,\mu_{K-m_K+j}))\right \|\\
\leq&M_{\alpha}\max_{0\leq j\leq m_K} c_{K-m_K+j}\left \|\u_{K-m_K+j}-\u^* \right \|
\leq\frac{2M_{\alpha}}{1-\bar{c}}\|\u_0-\u^*\|,
\end{aligned}$$
where the second equality follows from $\mathcal{G}(\u^*,\mu_{K-m_K+j})=G(\u^*)=\u^*$ and $\sum_{j=0}^{m_K}\alpha_j^K=1$, the last inequality uses (\ref{eq11}) and $c_{K-m_K+j}\leq\bar{c}$.
Thus, $\u_{K+1}\in\mathcal{B}(\delta)$ by $\|\u_0-\u^*\| \leq\frac{1-\bar{c}}{2M_{\alpha}}\delta$. This completes the induction.

(iii)
For $\sum_{j=0}^{m_K}\alpha_j^K \u_{K-m_K+j}$, by \cref{ass2} and \cref{eq11}, we obtain that
$$\left\|\sum_{j=0}^{m_K}\alpha_j^K\u_{K-m_{K}+j}-\u^*\right\| \leq M_{\alpha}\max_{0\leq j\leq m_K}\|\u_{K-m_{K}+j}-\u^*\|
\leq\frac{2M_{\alpha}}{1-\bar{c}}\|\u_0-\u^*\|.$$
Then, $\sum_{j=0}^{m_K}\alpha_j^{K}\u_{K-m_{K}+j}\in\mathcal{B}(\delta)$ by $\|\u_0-\u^*\|\leq\frac{1-\bar{c}}{2M_{\alpha}}\delta$.
\end{proof}

\begin{theorem}\label{theorem3}
Let Assumption \ref{ass2} hold. If $\u_0$ is sufficiently close to $\u^*$, then the iteration $\{\u_k\}$ given by Smoothing Anderson(m) converges to $\u^*$ and it holds
\[\limsup_{k\rightarrow\infty}\left(\frac{\|F(\u_k)\|}{\|F(\u_0)\|}\right)^{{1}/{k}}\leq {c}, \quad
\limsup_{k\rightarrow\infty}\left(\frac{\|\u_{k}-\u^*\|}{\|\u_0-\u^*\|}\right)^{{1}/{k}}\leq {c}.\]
\end{theorem}
\begin{proof}
Let $\epsilon > 0$ be given satisfying $c + \epsilon < 1$. By the continuity of $\tau(\mu)$, there exists $\bar{\mu}\in(0,\hat{\mu}]$ such that $\tau(\mu)\leq {c}+\frac{\epsilon}{2}$ for all $\mu \in [0,\bar{\mu}]$. Then, by \cref{ass3}, ${\mathcal{G}}(\u,\mu)$ is a contractive mapping on $\mathcal{B}(\delta)$ with factor ${{c}}+\frac{\epsilon}{2}$ and $\mathcal{G}(\u^*,\mu)=\u^*$ for any $\mu\in(0,\bar{\mu}]$. We first set $\u_0\in\mathcal{B}(\delta_0)$ with $\delta_0$ in (\ref{u0}) and $\mu=\bar{\mu}$ defined as above.

We divide the proof into four steps,
where the first three steps are to prove the following two important inequalities:
\begin{gather}
\|{\mathcal{F}}_k\|\leq({c}+\epsilon)^{k}\|{\mathcal{F}}_0\|,\quad\forall k=0,1,\ldots, \label{eq-th1-7} \\
\|{\mathcal{F}}_k\| \leq\max_{0\leq j\leq m_{k-1}}\{\|{\mathcal{F}}_{k-1-m_{k-1}+j}\| \},\quad\forall k=1,2, \ldots. \label{eq13-1}
\end{gather}

\textbf{Step 1: Initialization.}
Since \cref{eq-th1-7} holds for $k=0$, from \cref{lemma-uk}, $\u_1\in\mathcal{B}(\delta)$ and $\mu_0,\mu_1\in(0,\bar{\mu}]$. When $k=1$, $m_{k-1} = m_0 = 0$, \cref{eq13-1} reduces to $\|{\mathcal{F}}_1\|\leq\|{\mathcal{F}}_0\|$.
Notice that $\u_1 = \mathcal{G}_0$ by \cref{remark1}, then
\begin{equation}\label{eq38}
\begin{aligned}
\|{\mathcal{F}}_1\|=&\|{\mathcal{G}}_1-\u_1\|=\|{\mathcal{G}}_1-{\mathcal{G}}_0\|
\leq\|{\mathcal{G}}_1-{\mathcal{G}}(\u_0,\mu_1)\|+\|{\mathcal{G}}(\u_0,\mu_1)-{\mathcal{G}}_0\|.
\end{aligned}
\end{equation}
Then, by $\mu_0, \mu_1\in(0, \bar{\mu}]$, we have
\[\|{\mathcal{G}}(\u_0,\mu_1)-{\mathcal{G}}_0\|=\|{\mathcal{G}}(\u_0,\mu_1)-{\mathcal{G}}(\u^*,\mu_1)+
{\mathcal{G}}(\u^*,\mu_0)-{\mathcal{G}}(\u_0,\mu_0)\|
\leq\frac{\beta\|\u_0-\u^*\|^2}{2\mu_1}+\frac{\beta\|\u_0-\u^*\|^2}{2\mu_0},\]
where the equality follows from $\mathcal{G}(\u^*,\mu_1) = \u^* = \mathcal{G}(\u^*,\mu_0)$ by \cref{ass3}-(iii), and the inequality is obtained by applying the estimate from \cref{ass3}-(iv) to $\mathcal{G}(\u_0,\mu_1)$ and $\mathcal{G}_0$, invoking $\mathcal{G}'(\u^*,\mu_0) = \mathcal{G}'(\u^*,\mu_1)$ from \cref{ass3}-(iii).
Then by \cref{eq38} and since $\mathcal{G}(\cdot, \mu_1)$ is a contraction on $\mathcal{B}(\delta)$, we see that
\begin{equation}\label{eq13}
\begin{aligned}
\|{\mathcal{F}}_1\|
\leq&({c}+\frac{\epsilon}{2})\|\u_1-\u_0\|+\frac{\beta\|\u_0-\u^*\|^2}{\min\{\mu_0,\mu_1\}}
=({c}+\frac{\epsilon}{2})\|\mathcal{F}_0\|+\frac{\beta\|\u_0-\u^*\|^2}{\min\{\mu_0,\mu_1\}}.
\end{aligned}
\end{equation}

Recalling the definition of $\mu_0$ and $\mu_1$, and by (\ref{eq-FS}), we have
\begin{equation}\label{mu1}
\mu_1=\frac{1}{\sqrt{\|F(\u_0)\|}}\|{\mathcal{F}}_0\| \leq\frac{2}{1-c}\sqrt{\|F(\u_0)\|}
=\frac{2}{1-c}\mu_0. 
\end{equation}

Then, 
\[\begin{aligned}
\frac{\beta\|\u_0-\u^*\|^2}{\min\{\mu_0,\mu_1\}}\leq\frac{2\beta\|\u_0-\u^*\|^2}{(1-c)\mu_1}
\leq&\frac{2\beta}{(1-c)(1-{c}-\frac{\epsilon}{2})^2\mu_1}\|\mathcal{F}_0\|^2\\
\leq&\frac{2\beta}{(1-c)(1-{c}-\frac{\epsilon}{2})^2}\sqrt{\|F(\u_0)\|}\|\mathcal{F}_0\| \leq\frac{\epsilon}{2}\|\mathcal{F}_0\|,
\end{aligned}\]
where the first inequality follows from \eqref{mu1}, the second uses the first relation in \eqref{eqFS1}, the third inequality follows from $\|\mathcal{F}_0\| =\mu_1\sqrt{\|F(\u_0)\|}$ and the last inequality holds by reducing $\|\u_0-\u^*\|$ if necessary so that $\sqrt{\|F(\u_0)\|}\leq\frac{(1-c)(1-c-\epsilon/2)^2\epsilon}{4\beta}$.
The above inequality together with \cref{eq13} gives \cref{eq-th1-7}, and subsequently implies \cref{eq13-1} for $k=1$.

Then, we suppose that \cref{eq-th1-7} and \cref{eq13-1} hold for $0\leq k\leq K$ and we will establish them for $k=K+1$ with $K\geq1$. Since (\ref{eq-th1-7}) implies (\ref{eq3}), if (\ref{eq-th1-7}) holds for $k=0,1,\ldots, K$, from
Lemma \ref{lemma-uk}, we have that $\u_{k}\in\mathcal{B}(\delta)$ and $\mu_k\in(0,\bar{\mu}]$, $\forall 0\leq k\leq K+1$,
which implies
\begin{equation}\label{eq30}
\mathcal{G}(\u^*,\mu_k)=G(\u^*)=\u^*,\quad \forall 0\leq k\leq K+1.
\end{equation}

\textbf{Step 2: Relationship between \bm{$\mu_K$} and \bm{$\mu_{K+1}$}.}
When $K\geq1$, by (\ref{eq-u-s}) and the supposition of (\ref{eq13-1}) for $k=K$, we get
\[\begin{aligned}
\mu_{K+1}=\frac{1}{\sqrt{\|F(\u_0)\|}}\max_{0\leq j\leq m_K}
\left \|{\mathcal{F}}_{K-m_K+j} \right \|
\leq\frac{1}{\sqrt{\|F(\u_0)\|}}\max_{0\leq j\leq m_{K-1}}
\left \|{\mathcal{F}}_{K-1-m_{K-1}+j}\right \|=\mu_K.
\end{aligned}\]

\textbf{Step 3: Induction for (\ref{eq-th1-7}) and (\ref{eq13-1}).} 
Note that
\begin{equation}\label{eq-th1-1}
\begin{aligned}
\|{\mathcal{F}}_{K+1}\|=\|{\mathcal{G}}_{K+1}-{\u}_{K+1}\| \leq\|A_K\|+\|B_K\|,
\end{aligned}
\end{equation}
where
\[A_{K}={\mathcal{G}}_{K+1}-{\mathcal{G}} \left(\sum\nolimits_{j=0}^{m_K}\alpha_j^K{\u}_{K-m_K+j},\mu_{K+1}\right),\,\,
B_K={\mathcal{G}}\left (\sum\nolimits_{j=0}^{m_K}\alpha_j^K{\u}_{K-m_K+j},\mu_{K+1}\right )-{\u}_{K+1}.\]

Since ${\mathcal{G}}(\u,\mu_{K+1})$ is contractive on $\mathcal{B}(\delta)$ with factor $c_{K+1}\leq {c}+\frac{\epsilon}{2}$, the estimate of $\|A_K\|$ is the same as in \cite{BCK,ChenKelley2015,TothKelley2015}, i.e., 
\begin{equation}\label{eq-th1-9}
\|A_{K}\|
\leq{c}_{K+1}\left \|{\u}_{K+1}-\sum\nolimits_{j=0}^{m_K}\alpha_j^K{\u}_{K-m_K+j}\right \|
\leq{c}_{K+1}\|{\mathcal{F}}_{K}\| \leq({c}+\frac{\epsilon}{2})\|{\mathcal{F}}_{K}\|,
\end{equation}
where we use the condition of $\alpha_j^K$ in Assumption \ref{ass2}-(i).

By \cref{lemma-uk}-(iii), we have $\sum\nolimits_{j=0}^{m_K}\alpha_j^K{\u}_{K-m_K+j}\in\mathcal{B}(\delta)$. Using \cref{ass3}-(iv) to ${\mathcal{G}}_{K-m_K+j}$ and ${\mathcal{G}}\left (\sum\nolimits_{j=0}^{m_K}\alpha_j^K{\u}_{K-m_K+j},\mu_{K+1}\right )$, and by $\sum_{j=0}^{m_K}\alpha_j^K=1$, we have
$$
\begin{aligned}
\mathcal{G}\left(\sum_{j=0}^{m_K}\alpha_j^K{\u}_{K-m_K+j},\mu_{K+1}\right) &= \mathcal{G}(\u^*,\mu_{K+1})+\sum_{j=0}^{m_K}\alpha_j^K\mathcal{G}'(\u^*,\mu_{K+1})\left({\u}_{K-m_K+j}-\u^*\right)\\
&+\Delta_{\sum_{j=0}^{m_K}\alpha_j^K{\u}_{K-m_K+j},\mu_{K+1}}
\end{aligned}
$$
and
$$
{\mathcal{G}}_{K-m_K+j} = \mathcal{G}(\u^*,\mu_{K-m_K+j})+\mathcal{G}'(\u^*,\mu_{K-m_K+j})\left({\u}_{K-m_K+j}-\u^*\right) + \Delta_{{\u}_{K-m_K+j},\mu_{K-m_K+j}}.
$$

Combining the above two equations with \cref{ass3}-(iii) and $\sum_{j=0}^{m_K}\alpha_j^K=1$, we obtain
\begin{equation}\label{eq37}
\begin{aligned}
\|B_{K}\|
&= \left\| {\mathcal{G}}\left (\sum\nolimits_{j=0}^{m_K}\alpha_j^K{\u}_{K-m_K+j},\mu_{K+1}\right )- \sum\nolimits_{j=0}^{m_K}\alpha_j^K{\mathcal{G}}_{K-m_K+j} \right\| \\
&\leq\frac{\beta}{2\mu_{K+1}}\left\|\sum\nolimits_{j=0}^{m_K}\alpha_j^K{\u}_{K-m_K+j}-\u^*\right\|^2
+\sum\nolimits_{j=0}^{m_K}\frac{\beta|\alpha_j^K|}{2\mu_{K-m_K+j}}\|{\u}_{K-m_K+j}-\u^*\|^2.
\end{aligned}
\end{equation}

Note that
\begin{equation}\label{eq35}
\begin{aligned}
&\frac{1}{2\mu_{K+1}}\left \|\sum\nolimits_{j=0}^{m_K}\alpha_j^K{\u}_{K-m_K+j}-\u^*\right \|^2
\leq\frac{M_{\alpha}^2}{2\mu_{K+1}}\max_{0\leq j\leq m_K}\|{\u}_{K-m_K+j}-\u^*\|^2\\
\leq&\frac{M_{\alpha}^2}{2(1-\bar{c})^2\mu_{K+1}}\max_{0\leq j\leq m_K}\|{\mathcal{F}}_{K-m_K+j}\|^2
=\frac{M_{\alpha}^2\sqrt{\|F(\u_0)\|}}{2(1-\bar{c})^2}\max_{0\leq j\leq m_K}\|{\mathcal{F}}_{K-m_K+j}\|,
\end{aligned}
\end{equation}
where the first inequality uses \cref{ass2}-(iii), the second follows from $c_{K-m_K+j}\leq \bar{c}$, the first inequality in \eqref{eqFS1} and Lemma \ref{lemma-uk}-(ii), and the last equality employs the definition of $\mu_{K+1}$.

Similar to the analysis of (\ref{eq35}), by $\mu_1\leq\frac{2}{1-c}\mu_0$ and since $\mu_{k}$ is nonincreasing for $k=1,\ldots,K+1$,
we get
\begin{equation}\label{eq36}
\begin{aligned}
\frac{1}{2\mu_{K-m_K+j}}\|{\u}_{K-m_K+j}-\u^*\|^2
&\leq\frac{1}{(1-\bar{c})^2(1-c)\mu_{K+1}}\|{\mathcal{F}}_{K-m_K+j}\|^2 \\
&\leq\frac{\sqrt{\|F(\u_0)\|}}{(1-\bar{c})^2(1-c)}\|{\mathcal{F}}_{K-m_K+j}\|.
\end{aligned}
\end{equation}

Combining (\ref{eq-th1-1})-(\ref{eq36}), we have
\begin{equation}\label{eq-th1-8}
\|{\mathcal{F}}_{K+1}\|
\leq({c}+\frac{\epsilon}{2})\|{\mathcal{F}}_{K}\|
+\lambda\sqrt{\|F(\u_0)\|} \max_{0\leq j\leq m_K}\|{\mathcal{F}}_{K-m_K+j}\|
\end{equation}
with $\lambda=\frac{M_{\alpha}^2\beta }{2(1-\bar{c})^2}+
\frac{M_{\alpha}\beta }{(1-\bar{c})^2(1-c)}$.
Hence, (\ref{eq13-1}) holds for $k=K+1$ by further reducing $\|\u_0-\u^*\|$ if necessary so that $\sqrt{\|F(\u_0)\|}\leq\frac{1-{{c}}-\frac{\epsilon}{2}}{\lambda}$.

Recalling (\ref{eq-th1-8}) and by the supposition in \eqref{eq-th1-7} for $0\leq k\leq K$,
we have
\[
\|{\mathcal{F}}_{K+1}\| \leq({c}+\epsilon)^{K+1}\|{\mathcal{F}}_{0}\|
\left(\frac{{c}+\frac{\epsilon}{2}}{{c}+\epsilon}+\lambda\sqrt{\|F(\u_0)\|}({c}+\epsilon)^{-m-1}\right)
\leq({c}+\epsilon)^{K+1}\|{\mathcal{F}}_{0}\|,
\]
which holds as $\|\u_0-\u^*\|$ is sufficiently small so that
$$\sqrt{\|F(\u_0)\|}
\leq\frac{\epsilon({c}+\epsilon)^{m}}{2\lambda}.$$
Therefore, the estimation in (\ref{eq-th1-7}) holds for $k=K+1$.

\textbf{Step 4: Calculation of \bm{$\|{F}({u}_{k+1})\|$}.}
From \cref{ass3}-(i), we have
\begin{equation}\label{eq4}
\|{\mathcal{F}}_{k+1}\|
=\|{\mathcal{G}}_{k+1}-{\u}_{k+1}\|
\geq \|{G}({\u}_{k+1})-{\u}_{k+1}\|-\|{\mathcal{G}}_{k+1}-{G}({\u}_{k+1})\| \geq \|F(\u_{k+1})\|-\kappa\mu_{k+1}.
\end{equation}
Thanks to \cref{eq-u-s}, \cref{eq-FS}, \cref{eq-th1-7} and \cref{eq4}, one has
\begin{equation}\label{eq9-s}
\begin{aligned}
\|F(\u_{k+1})\|\leq&\|{\mathcal{F}}_{k+1}\|
+\kappa\mu_{k+1}
\leq({c}+\epsilon)^{k+1}\|{\mathcal{F}}_{0}\|
+\frac{\kappa}{\sqrt{\|F(\u_0)\|}}({c}+\epsilon)^{k-m}\|{\mathcal{F}}_{0}\|  \\
\leq&({c}+\epsilon)^{k+1}
\left(1+\frac{\kappa}{({c}+\epsilon)^{m+1}\sqrt{\|F(\u_0)\|}}\right)\left(\frac{2}{1-c}\right)\|F(\u_0)\|,
\end{aligned}\end{equation}
where the last inequality uses (\ref{eq-FS}).

By $$\lim_{k\rightarrow\infty}
\left[\left(1+\frac{\kappa}{({c}+\epsilon)^{m+1}\sqrt{\|F(\u_0)\|}}\right)\left(\frac{2}{1-c}\right)\right]^{\frac{1}{k+1}}=1,$$ (\ref{eq9-s}) implies
$$\limsup_{k\rightarrow\infty}\left(\frac{\|F(\u_{k+1})\|}{\|F(\u_0)\|}\right)^{\frac{1}{k+1}}\leq {c}+\epsilon.$$

Since we can restart the proof by letting $\u_0$ sufficiently close to $\u^*$ to reduce $\epsilon$ further, due to the arbitrariness of $\epsilon$, we complete the proof of the r-linear convergence of $\|F(\u_k)\|$ with factor ${c}$.

From \cref{eqF1} and $\lim_{k\rightarrow\infty}\left(\frac{1-c}{1+c}\right)^{1/k}=1$, we can also obtain the r-linear convergence of $\|\u_k-\u^*\|$ with factor ${c}$, i.e., 
\[\limsup_{k\rightarrow\infty}\left(\frac{\|\u_{k}-\u^*\|}{\|\u_0-\u^*\|}\right)^{\frac{1}{k}}\leq {c}.\]

\end{proof}

\begin{lemma}\label{lemma1}
With the conditions in Theorem \ref{theorem3}, the sequence $\{\mu_k\}$ defined in (\ref{eq-u-s}) satisfies $\mu_{k+1}\leq\mu_k$ for $k\geq1$ and $\lim_{k\rightarrow\infty}\mu_k=0$.
\end{lemma}
\begin{proof}
The first result is proved in the step 2 of the proof of \cref{theorem3}.
Based on the estimation in (\ref{eq-th1-7}) and the definition of $\mu_{k+1}$ in (\ref{eq-u-s}), we have
$$\lim_{k\rightarrow\infty}\mu_{k+1}
\leq\lim_{k\rightarrow\infty}\frac{1}{\sqrt{\|F(\u_0)\|}}({c}+\epsilon)^{k-m}\|{\mathcal{F}}_0\|=0.$$
\end{proof}

\begin{remark}
    By applying the smoothing function of $G$, we can develop the Smoothing EDIIS(m) by replacing the residuals in \cref{EDIIS_alpha} with the smoothing residuals. Similarly, Smoothing EDIIS(m) is also r-linearly convergent with factor no larger than $c$. Moreover, \cref{ass2} can be weakened for the Smoothing EDIIS(m) algorithm, with condition (ii) being converted to the constraints in \eqref{EDIIS_alpha}, and condition (iii) removed.
\end{remark}

\subsection{q-linear convergence of Smoothing Anderson(1)}\label{section2.3}
In this subsection, we focus on the study of Smoothing Anderson(1). For $m=1$, when $\mathcal{F}_k \neq \mathcal{F}_{k-1}$, similar to the expression of Anderson(1) in \cite{BCK}, 
the solution $\alpha^k_0$ (denoted by $\alpha^k$) to \cref{s-anderson-m}
can be expressed in closed form
\begin{equation}\label{eq-alpha-s}
\alpha^k=\frac{\mathcal{F}_k^T(\mathcal{F}_k-\mathcal{F}_{k-1})}
{\|\mathcal{F}_k-\mathcal{F}_{k-1}\|^2},
\end{equation}
and notice that \cref{eq-u-s} is reduced to
\begin{equation}\label{eq-mu1-s}
\u_{k+1}=(1-\alpha^k)\mathcal{{G}}_k+\alpha^k\mathcal{{G}}_{k-1},\,\,\mu_{k+1}=\frac{1}{\sqrt{\|F(\u_0)\|}}\max \left \{\|{\mathcal{F}}_{k-1}\|,\|{\mathcal{F}}_{k}\|\right \}.
\end{equation}
As for the situation that $\mathcal{F}_k = \mathcal{F}_{k-1}$, the parameter $\alpha^k$ is free, and we can set $\alpha^k = 0$ and $\u_{k+1} = \mathcal{G}_k$ for simplicity.

Though $\mathcal{G}(\u,\mu)$ is Lipschitz continuously differentiable with respect to $\u$ for any fixed $\mu$, the q-linear convergence of $\mathcal{F}_k$ for Smoothing Anderson(1) can not be similarly derived by the proof of \cite[Corollary 2.5]{TothKelley2015}. The main reason is the different $\mu_k$ in each iteration.

\begin{theorem}\label{theorem4-s}
If $\u_0$ is sufficiently close to $\u^*$, the smoothing residual function $\|\mathcal{F}_k\|$ is q-linearly convergent to $0$, i.e.,
\[\limsup_{k\rightarrow\infty}\frac{\|\mathcal{F}_{k+1}\|}{\|\mathcal{F}_{k}\|}\leq {c}.\]
\end{theorem}
\begin{proof}
For a given $\epsilon\in(0,1-{c})$, we set $\u_0$ sufficiently close to $\u^*$.
% such that (\ref{eq-th1-7}) and (\ref{eq13-1}) hold. 
We will first show that
\begin{equation}\label{th4-1-s}
\|\mathcal{F}_{k+1}\| \leq({c}+\epsilon)\|\mathcal{F}_{k}\|\quad \mbox{for all $k\geq0$}.
\end{equation}
We use induction on $k$ to prove (\ref{th4-1-s}) and assume it holds for $0\leq k\leq K-1$, which clearly holds for
$K=1$ by Step 1 of the proof of \cref{theorem3}.

From the supposition that (\ref{th4-1-s}) holds for $k=K-1$, we have
$
\|\mathcal{F}_K-\mathcal{F}_{K-1}\|\geq\|\mathcal{F}_{K-1}\|-\|\mathcal{F}_{K}\| \geq(1-{c}-\epsilon)
\|\mathcal{F}_{K-1}\|.$
Similar to the analysis in \cite[pp.~S5, eq. (1.11)]{BCK}, by \eqref{eq-alpha-s} and (\ref{th4-1-s}), one has
\begin{equation}\label{alpha-s}
|\alpha^K| = \frac{|\mathcal{F}_K^T(\mathcal{F}_K-\mathcal{F}_{K-1})|}
{\|\mathcal{F}_K-\mathcal{F}_{K-1}\|^2}
\leq \frac{\|\mathcal{F}_K\|}{\|\mathcal{F}_K-\mathcal{F}_{K-1}\|} \leq \frac{({c}+\epsilon)\|\mathcal{F}_{K-1}\|}{(1-{c}-\epsilon)\|\mathcal{F}_{K-1}\|} = \frac{c+\epsilon}{1-c-\epsilon}.
\end{equation}
Then, applying the same estimate for each $1\le k \le K$, we have
$|1-\alpha^k|+|\alpha^k|\leq M_{\alpha}:=\frac{1+{c}+\epsilon}{1-{c}-\epsilon}$. The induction hypothesis also implies \eqref{eq-th1-7} for $0\le k\le K$
and \eqref{eq13-1} for $1\le k\le K$.
By \cref{ass3} and Lemma \ref{lemma-uk}, one has
\[\u_{K+1}\in\mathcal{B}(\delta)\quad\mbox{and}\quad \boldsymbol{\omega} = (1-\alpha^K)\u_K+\alpha^K\u_{K-1}\in\mathcal{B}(\delta),\]
which implies ${\mathcal{G}}(\u,\mu_k)$ is a contractive mapping on $\mathcal{B}(\delta)$ with factor ${c}+\frac{\epsilon}{2}$ and $\mathcal{G}(\u^*,\mu_k)=\u^*$ for any $k=0,...,K+1$.

Following the proof of (\ref{eq-th1-9}) and (\ref{eq37}), and noting that $\mu_1 \leq \frac{2}{1-c}\mu_0$ and $\mu_{k+1}\leq\mu_k$, $\forall k=1,\ldots,K$, it holds that
\begin{equation}\label{th2.5-1}
    \begin{aligned}
    \|\mathcal{F}_{K+1}\|&\leq \left({c}+\frac{\epsilon}{2}\right)\left\|\mathcal{F}_K \right\|+ \frac{\beta}{2\mu_{K+1}}\left\|\boldsymbol{\omega}-\u^*\right\|^2+\frac{\beta |1-\alpha^K|}{2\mu_{K}}\left\|\u_K-\u^*\right\|^2 + \frac{\beta|\alpha^K|}{2\mu_{K-1}}\left\|\u_{K-1}-\u^*\right\|^2 \\
    &\leq \left({c}+\frac{\epsilon}{2}\right)\left\|\mathcal{F}_K \right\| + \xi_K,
    \end{aligned}
\end{equation}
where
\begin{equation}\label{th4-3-s}
\xi_K=\frac{\beta}{(1-c)\mu_{K+1}}\left((|1-\alpha^K|+1)|1-\alpha^K|\|\u_K-\u^*\|^2+
(|\alpha^K|+1)|\alpha^K|\|\u_{K-1}-\u^*\|^2\right),
\end{equation}
and the second term follows from $\frac{1}{2}\|\boldsymbol{\omega} - u^* \|^2\leq |1-\alpha^K|^2 \|\u_K-\u^*\|^2 + |\alpha^K|^2 \|\u_{K-1}-\u^*\|^2$.

By (\ref{eqFS1}) and the definition of $\mu_{k+1}$ in \cref{eq-u-s}, we get
\[\frac{1}{\mu_{K+1}}
\|\u_{K-j}-\u^*\|^2\leq
\frac{1}{\mu_{K+1}(1-{c}-\epsilon)^2}\|\mathcal{F}_{K-j}\|^2
\leq\frac{\sqrt{\|{F}(\u_{0})\|}}{(1-{c}-\epsilon)^2}\|\mathcal{F}_{K-j}\|,\quad j=0,1.
\]
Together with \eqref{alpha-s}, we have
\[
|\alpha^K|\|\mathcal{F}_{K-1}\|=\frac{|\mathcal{F}_K^T(\mathcal{F}_K-\mathcal{F}_{K-1})|}{\|\mathcal{F}_K-\mathcal{F}_{K-1}\|^2}\|\mathcal{F}_{K-1}\|
\leq\frac{1}{1-{c}-\epsilon}\|\mathcal{F}_K\|,\]
which implies
\begin{equation*}{
\xi_K
\leq\frac{(M_{\alpha}+1)\beta\sqrt{\|{F}(\u_{0})\|}}{(1-{c}-\epsilon)^2(1-c)}\left (M_{\alpha}+\frac{1}{1-{c}-\epsilon} \right )
\|\mathcal{F}_{K}\|. }
\end{equation*}

Hence, (\ref{th4-1-s}) holds for $k=K$
by reducing $\|\u_0-\u^*\|$ so that
$\xi_K\leq\frac{\epsilon}{2}\|\mathcal{F}_{K}\|$.
Since $\epsilon$ is arbitrary, we finish the proof of this statement.
\end{proof}

From (\ref{th4-1-s}), when $\u_0$ is sufficiently close to $\u^*$, we can easily find that $\|\mathcal{F}_k\|$ is monotonically decreasing, and by (\ref{eq-mu1-s}), we know that
\[\mu_{k+1}=\frac{1}{\sqrt{\|F(\u_0)\|}}\|\mathcal{F}_{k-1}\|,\quad \mbox{for all $k\geq1$}.\]

In Smoothing EDIIS(1), $\alpha^k$ is chosen as a minimizer of the following optimization problem
\begin{equation*}
\min \,\, \frac{1}{2} \left\|(1-\alpha)\mathcal{F}_k+\alpha \mathcal{F}_{k-1}\right\|^2, \quad {\rm s.t.} \; 0\le  \alpha\le 1. 
\end{equation*}
This is a convex optimization problem and its first order optimality condition is
$$\left(\mathcal{F}_k^T(\mathcal{F}_{k-1}-\mathcal{F}_k)+\alpha^k\|\mathcal{F}_{k-1}-\mathcal{F}_k\|^2\right)(\alpha -\alpha^k)\ge 0, \quad {\rm for } \quad \alpha\in [0,1].$$
Hence the solution $\alpha^k$ can be expressed as follows:
\[
\alpha^k = P_{[0,1]}\left(\frac{\mathcal{F}_k^T(\mathcal{F}_k-\mathcal{F}_{k-1})}{\|\mathcal{F}_{k-1}-\mathcal{F}_k\|^2}\right),
\]
where $P_{[0,1]}$ is the projection operator on $[0,1]$.

Following a similar analysis of \cref{theorem4-s}, we only need to check the values of $\xi_K$ with $\alpha_K=0$ and $\alpha^K=1$, and then the results also hold for the sequence $\{\u_k\}$ generated by
Smoothing EDIIS(1).
\begin{corollary}\label{corollary2-s}
Suppose that the assumptions of \cref{theorem4-s} hold. The results in \cref{theorem4-s} also hold for the sequence $\{\u_k\}$ generated by
Smoothing EDIIS(1).
\end{corollary}

\begin{remark}\label{remark24}
As proved in Theorem \ref{theorem3}, the Smoothing Anderson($m$) algorithm with updated smoothing parameter is r-linearly convergent with factor $c$, which is the contractive factor of $G$. By \cite[Theorem 2.2]{bian2022anderson}, the Smoothing Anderson(m) algorithm with a fixed smoothing parameter enjoys an r-linear convergence rate, with a factor of $\tau(\mu)$, which is the contractive factor of $\mathcal{G}(\cdot,\mu)$. Here, $\tau(\mu)$ might be strictly greater than $c$. For example, if we consider $G(\u) = \max\{Q(\u),0\}$ as in \cref{ex1},  its smoothing function $\mathcal{G}(\cdot,\mu)$ is a contractive mapping on $\mathcal{B}(\delta)$ with factor $\tau(\mu) \leq c(1+\sqrt{\mu}/2)$. In the specific case where $Q(\u)=c\u$, $\tau(\mu)=c(1+\sqrt{\mu}/2)$.
Besides, compared to fixed $\mu$, the benefits of updating $\mu$ are evidenced by the experimental results in \cref{section3} (Fig. \ref{fig:jourtime}).
\end{remark}

\section{Numerical applications and examples}\label{section3}
In this section, we illustrate the performance of the proposed Smoothing Anderson(m) for nonsmooth fixed point problem \cref{eq1-0} by three applications. For simplicity, we refer to the Smoothing Anderson(m) algorithm as s-Anderson(m). All our experiments are performed in Python on a MacBook Pro (2.30GHz, 8.00GB of RAM). In practice, we always transform (\ref{s-anderson-m}) into an unconstrained optimization problem as follows:
\begin{equation}
    \bm{\gamma}^k \in \arg \min \limits_{\bm{\gamma} \in \mathbb{R}^{m_k}} \left\|\mathcal{F}_k-\sum\nolimits_{j=0}^{m_k-1}\gamma_j\left(\mathcal{F}_{k-m_k+j+1}-\mathcal{F}_{k-m_k+j}\right)\right\|
\end{equation}
and set
$$\u_{k+1}=\mathcal{G}_k-\sum\nolimits_{j=0}^{m_k-1}\gamma_j^k(\mathcal{G}_{k-m_k+j+1}-\mathcal{G}_{k-m_k+j}),$$
where $\bm{\gamma}^k = (\gamma_0^k, \dots, \gamma_{m_k-1}^k)^T$ and ${\gamma}_j^k = \sum_{i=0}^{j}\alpha_i^k$ for $j = 0, 1, \dots, m_k-1$. 

We stop the algorithms when
\begin{equation}\label{stop}
\frac{\|F(\u_k)\|}{\|F(\u_0)\|}\leq {\rm tol} \quad {\rm or } \quad k\geq k_{\max},
\end{equation}
for a given tolerance ${\rm tol}$ and maximum iteration $k_{\max}$. It should be noted that we also apply the above stopping criterion to s-Anderson(m), as illustrated in \cref{algo-s}, while Anderson(m) is outlined in \cref{algo-o}. For simplicity, we consistently represent the relative residual norm $\|F(\u_k)\|/\|F(\u_0)\|$ as $\|F_k\|/\|F_0\|$ throughout the following examples. The label `Anderson(m), $\mu$' indicates the application of Anderson(m) to the smoothing function $\mathcal{G}$ with a fixed smoothing parameter $\mu$. In particular, when $\mu = 0$, it reduces to Anderson(m). “$-$" in tables means failing to complete the test in the given stopping criterion. In the tables below, the bold entries denote the best performance among all compared algorithms, while the underlined entries indicate the better one between Anderson(m) and s-Anderson(m) under the same experimental parameters. To reduce the influence of random variations, each experiment is repeated 10 times, and the averaged results are reported as the final values. 

In various applied settings, the fixed point function frequently admits a smoothing function that satisfies \cref{ass3}. This holds not only for the  `max' operator but also for its different combinations and some other functions, such as the absolute value function. In this section, we aim to illustrate the flexibility of constructing appropriate smoothing functions in practical problems through three examples, thereby demonstrating the effectiveness and broad applicability of s-Anderson(m).

\subsection{{Iterative shrinkage-thresholding algorithm}}\label{subsec1}

Consider the following minimization problem:
\begin{equation}\label{l1_reg}
    \min \limits_{\u\in \mathbb{R}^n} \phi(\u) = f(\u)+h(\u),
\end{equation}
where $f: \mathbb{R}^n \rightarrow \mathbb{R}$ is convex, $L$-smooth, i.e., $\|\nabla f(\u) - \nabla f(\v)\| \leq L\|\u-\v\|$, and $h:\mathbb{R}^n\rightarrow \mathbb{R}\cup\{+\infty\}$ is a proper, closed and convex function. When the proximal operator of $h$ can be computed explicitly, a classical method for solving (\ref{l1_reg}) is the proximal gradient algorithm (PGA): 
\begin{equation}\label{PGA}
    \u_{k+1} = \operatorname{prox}_{\alpha h}(\u_k - \alpha \nabla f(\u_k)),
\end{equation}
where $\operatorname{prox}_{\alpha h}(\u) = \arg \min \nolimits_{\bm{x}\in \mathbb{R}^n} \{h(\bm{x})+\frac{1}{2 \alpha}\|\bm{x}-\u\|^2\}$ and $\alpha \in (0,2/L)$ is a suitable stepsize. 
In many applications including machine learning, $h$ is often chosen as the $\ell_1$ regularization, i.e., $h(\u)=\lambda \|\u\|_1$, where $\lambda$ is the positive regularization parameter. 
In this case, PGA reduces to the iterative shrinkage-thresholding algorithm (ISTA), i.e.,
\begin{equation}\label{ISTA}
    \u_{k+1} = S_{\alpha \lambda}(\u_k - \alpha \nabla f(\u_k)) \quad \mbox{and} \quad S_{\alpha \lambda}(\u)_i = \operatorname{sign}(u_i)(|u_i|-\alpha \lambda)_{+}, \quad i=1, \ldots, n.
\end{equation}
Obviously, $S_{\alpha \lambda}$ is a nonexpansive operator. We define the fixed point problem as
\begin{equation}\label{ISTA_fixed}
    \u = G(\u) := S_{\alpha \lambda}(\u - \alpha \nabla f(\u)).
\end{equation}

With the max operator, $S_{\alpha \lambda}$ can be rewritten in an alternative form, i.e., $S_{\alpha \lambda}(\u) = \max\{\u - \alpha \lambda, 0\} - \max\{-\u - \alpha \lambda, 0\}.$ By setting $Q:\mathbb{R}^n\rightarrow \mathbb{R}^{n}$ with $Q(\u):=\u-\alpha \nabla f(\u)$, using $\phi$ in \cref{sm1}, we can set the smoothing function by
$\mathcal{G}(\u,\mu)=\Phi(Q(\u)-\alpha \lambda,\mu) - \Phi(-Q(\u)-\alpha \lambda,\mu)$, 
where $\Phi(\v,\mu)=(\phi(v_1,\mu),\ldots,\phi(v_n,\mu))^T$ for $\v=(v_1, \ldots, v_n)^T\in\mathbb{R}^n$.

As for the contraction of $G$,
we add a new assumption as follows:
\begin{assumption}\label{str_mon}
    The mapping $\nabla f$ is strongly monotone, i.e., there exists a positive constant $\tau$ such that for all $\u, \v \in D$, it holds that
    $(\nabla f(\u) - \nabla f(\v))^ \mathrm{T}(\u-\v) \geq \tau \|\u - \v\|^2.$
\end{assumption}

For $\u, \v \in D$, since $\nabla f$ and $S_{\alpha\lambda}$ are Lipschitz continuous, by \cref{str_mon}, we have
\begin{equation*}
    \begin{aligned}
    \|S&_{\alpha \lambda}(\u - \alpha \nabla f(\u)) - S_{\alpha \lambda}(\v - \alpha \nabla f(\v))\|^2 \leq \|\u - \alpha \nabla f(\u) - \v + \alpha \nabla f(\v) \|^2 \\ 
    =& \|\u - \v\|^2 + \alpha^2 \|\nabla f(\u) - \nabla f(\v)\|^2 -2\alpha(\u-\v)^ \mathrm{T}(\nabla f(\u) - \nabla f(\v)) \\
    \leq& (1+\alpha^2L^2 -2\alpha \tau)\|\u- \v\|^2.
    \end{aligned}
\end{equation*}
Obviously, if $\alpha \in (0, \min\{2/L, 2 \tau/L^2\})$, then $1+\alpha^2L^2 -2\alpha \tau<1$, which means that $G$ in (\ref{ISTA_fixed}) is a contraction mapping. In particular, when $\nabla f(\u) = \bm{M} \u + \b$ with a symmetric positive definite matrix $\bm{M}\in \mathbb{R}^{n \times n}$ and $\b \in \mathbb{R}^n$, the contraction condition is satisfied for any $\alpha \in (0, 2/L)$.

\begin{remark}
    For $h(\u)=\lambda\|\u\|_2$, its proximal operator ${P}(\u)=\operatorname{prox}_{\alpha h}(\u)$ is nonsmooth. 
    Based on the smoothing function $\phi(t, \mu)$ for $\max\{t,0\}$ defined in \eqref{sm1}, 
    we can construct a smoothing function $\mathcal{P}(\u,\mu)$ of $P$, where
    \begin{equation}\label{prox_2}
            P(\u) =\left\{
                \begin{aligned}
                    &\max\left\{\|\u\| -\alpha\lambda, 0\right\}\frac{\u}{\|\u\|} &&\mbox{if $\u\neq 0$}\\
                    &0 &&\mbox{if $\u=0$}
                \end{aligned}
                \right.
        \quad
        \mbox{and}
        \quad
            \mathcal{P}(\u,\mu) =\left\{
                \begin{aligned}
                    &\phi(\|\u\| -\alpha\lambda, \mu)\frac{\u}{\|\u\|} &&\mbox{if $\u\neq 0$}\\
                    &0 &&\mbox{if $\u=0$.}
                \end{aligned}
                \right.
    \end{equation}
    Moreover, for group-LASSO problems, a similar structure arises due to the group partition. Using \cref{prox_2}, we can construct a group-wise smoothing function, and then apply the proposed Smoothing Anderson(m) to solve the problem under contractive conditions.
\end{remark}

\begin{example}
Elastic net regression is an extension of Lasso regression, which is a combination of the $\ell_1$ (Lasso regression) and $\ell_2$ (Ridge regression) penalties, and has been widely used in domains with massive datasets, such as machine learning. It can improve the stability in high dimensional data and has become an important tool in the analysis of these datasets \cite{zou2005regularization}. 

Consider the following elastic net regression (ENR) problem:
\begin{equation}\label{enr}
    \min_{\u \in \mathbb{R}^n}  \frac{1}{2}\|\A\u - \b\|^2 + \lambda \left(\frac{1-\beta}{2}\|\u\|^2 + \beta \|\u\|_1 \right)
\end{equation}
where $\A \in \mathbb{R}^{M \times n}$, $\b\in \mathbb{R}^M$. In this example, we choose $\beta = \frac{1}{2}$ and $\lambda = 0.001\|\A^T\b\|_{\infty}$, where $2\|\A^T\b\|_{\infty}$ is the smallest value of $\lambda$ to guarantee that problem (\ref{enr}) only has the zero solution \cite{o2016conic}. 

When we use ISTA to solve the ENR problem in (\ref{enr}), the iteration scheme is as follows: 
\begin{equation}
    \u_{k+1} = S_{\alpha \lambda/2} \left(\u_k - \alpha\left(\A^T(\A\u_k - \b) + \frac{\lambda}{2}\u_k\right) \right).
\end{equation}
Obviously, $S_{\alpha \lambda/2}\left(\u - \alpha (\A^T(\A\u - \b) + \frac{\lambda}{2}\u)\right)$ is contractive when $\alpha \in (0, 2/L)$ with $L = \|\A\|^2+\lambda/2$. It is worth noting that ISTA can be regarded as the original fixed point iteration, which is also known as Picard iteration.  

The data for this experiment is generated in a manner similar to the ENR example in \cite{zhang2020globally}. Specifically, we set $(M,n) = (500,1000)$, randomly generate the matrix $\A$, construct a sparse vector ${\bm{X}}$ with a predetermined sparsity level, and then form the observation vector $\b$ by adding noise to $\A \bm{X}$. The initial point $\u_0$ is chosen as a Gaussian random vector scaled by 10. We set the stepsize $\alpha = 1.8/L$ with $L = \|\A\|^2+\lambda/2$, ${\rm tol} = 10^{-15}$ and $k_{\max} = 10000$. 

First, we present the results of Anderson(m) and s-Anderson(m) with different values of $m$ for the ENR problem. The convergence of $\|F_k\|/\|F_0\|$ by Anderson(m) and s-Anderson(m) is shown in Fig. \ref{fig:enr1} and \ref{fig:enr2} when the sparsity of $\bm{X}$ is 0.1, from which we observe  that s-Anderson(m) outperforms Anderson(m) in both iteration numbers and running time. Moreover, the advantages of s-Anderson(m) over Anderson(m) for the ENR problems with different sparsity on the data are also illustrated in Fig. \ref{fig:enr3} and \ref{fig:enr4}. It can be observed that the gap between s-Anderson(m) and Anderson(m) decreases as sparsity increases, though s-Anderson(3) remains the most effective. Our analysis suggests that this phenomenon arises from the increased sparsity, which results in a greater proportion of non-zero elements in the fixed point $\u^*$. As the number of non-zero elements grows, by the definition of smoothing function, the effect of the smoothing function diminishes, thereby reducing the performance gap between s-Anderson(m) and Anderson(m).

Next, we design experiments to compare the s-Anderson(m) with some other well-known algorithms including FISTA \cite{beck2009fast}, Type-I Anderson Acceleration (AA1) and Stabilized type-I Anderson acceleration\footnote[1]{The code is available at \url{https://github.com/cvxgrp/nonexp_global_aa1}} (AA1-safe) \cite{zhang2020globally}, whose results are shown in \cref{table:enr} and Fig. \ref{fig:enr_1}. We set the stepsize $\alpha = 0.8/L$ for FISTA, the max-memory $m_{\max}=3$ for AA1 and AA1-safe and other hyper-parameters for AA1-safe follow the settings in \cite{zhang2020globally}, while other parameters are the same as in s-Anderson(m). \cref{table:enr} reports the iteration and computational time cost, where we set different tolerances ($10^{-6}/10^{-15}$) to compare the performance of different algorithms. It indicates that s-Anderson(m) introduces extra running time in every iteration compared with Anderson(m), but consumes less time overall, since s-Anderson(m) requires much fewer iterations. In comparison with AA1-safe, although s-Anderson(m) runs more iterations, it significantly reduces the computational time of each iteration, which makes s-Anderson(m) outperform AA1-safe in total running time.

Finally, Fig. \ref{fig:enrsmooth} illustrates the variation of the smoothing parameter, tending towards zero as $k\rightarrow \infty$, which is consistent with the results in \cref{lemma1}. 
\begin{figure}[t]
\centering 
\subfigure[Convergence of $\|F_k\|/\|F_0\|$ by Anderson(m) and s-Anderson(m) with sparsity=0.1]{
\label{fig:enr1}
\includegraphics[width=0.455\textwidth]{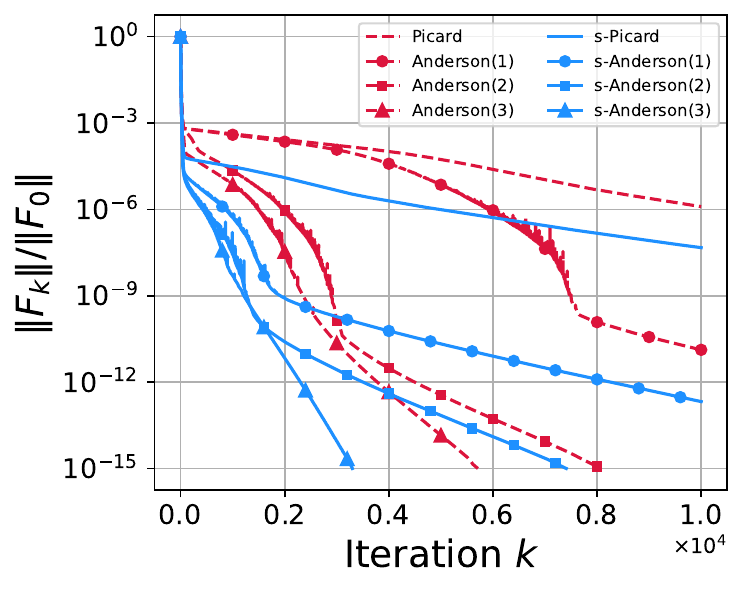}
}
\subfigure[{Convergence of $\|F_k\|/\|F_0\|$ versus time (seconds) with sparsity=0.1}]{
\label{fig:enr2}
\includegraphics[width=0.455\textwidth]{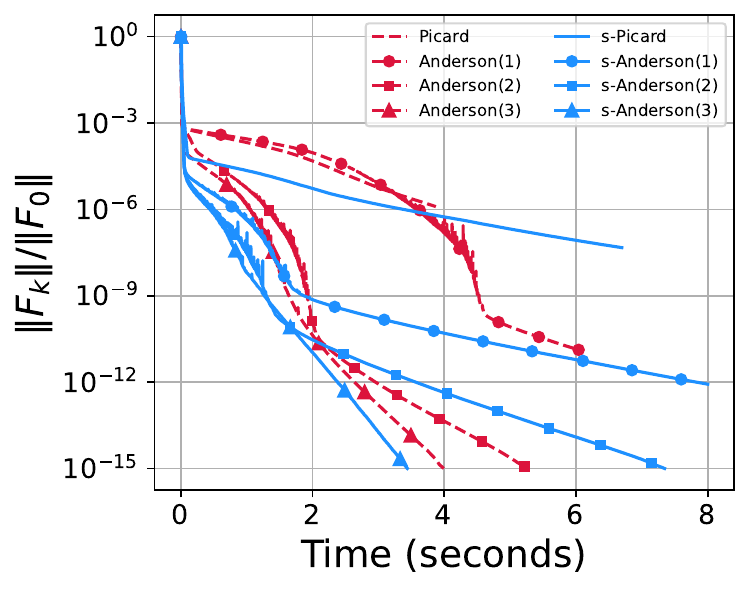}
}
\subfigure[Convergence of $\|F_k\|/\|F_0\|$ by Anderson(m) and s-Anderson(m) with sparsity=0.2]
{
\label{fig:enr3}
\includegraphics[width=0.455\textwidth]{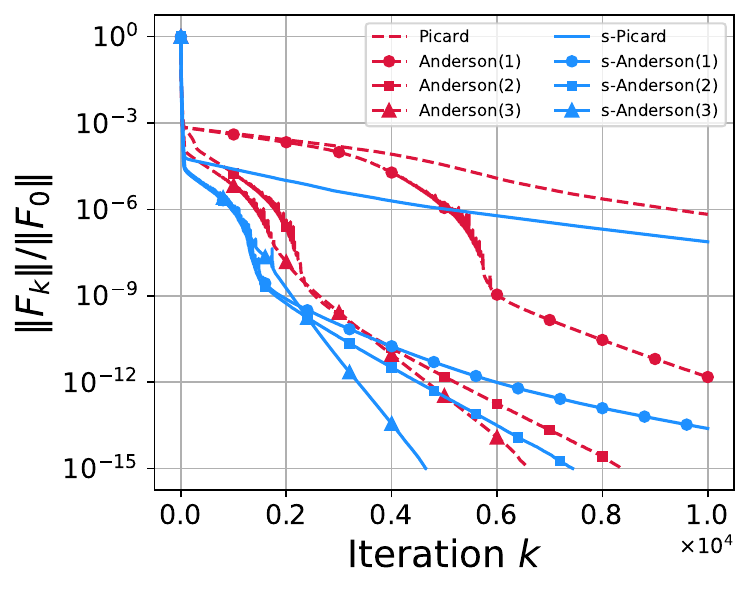}
}
\subfigure[Convergence of $\|F_k\|/\|F_0\|$ by Anderson(m) and s-Anderson(m) with sparsity=0.3]
{
\label{fig:enr4}
\includegraphics[width=0.455\textwidth]{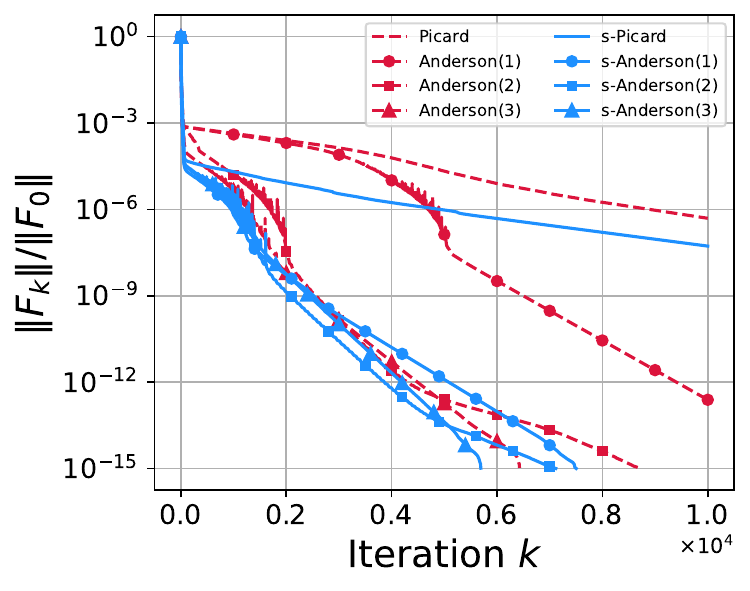}
}
\caption{Elastic net regression problem with different sparsity for ENR problem}
\label{fig:enr}
\vskip -0.4in
\end{figure}

\begin{figure}[ht]
    \centering
    \includegraphics[width=1\textwidth]{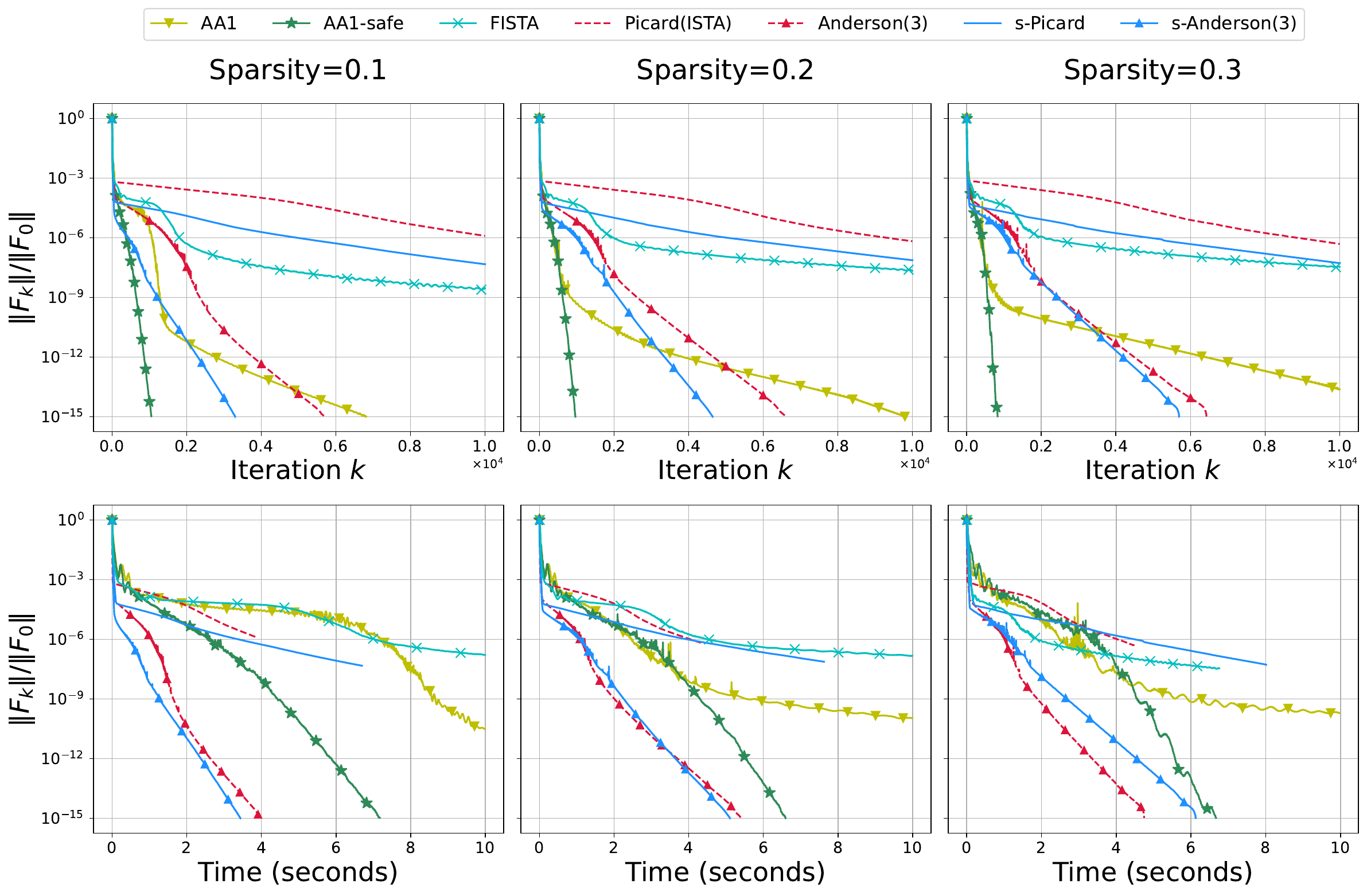}
    \caption{Comparisons of some algorithms for the ENR problem with different sparsities}
    \label{fig:enr_1}
    \vskip -0.2in
\end{figure}

\begin{figure}
    \centering
    \includegraphics[width=1\textwidth]{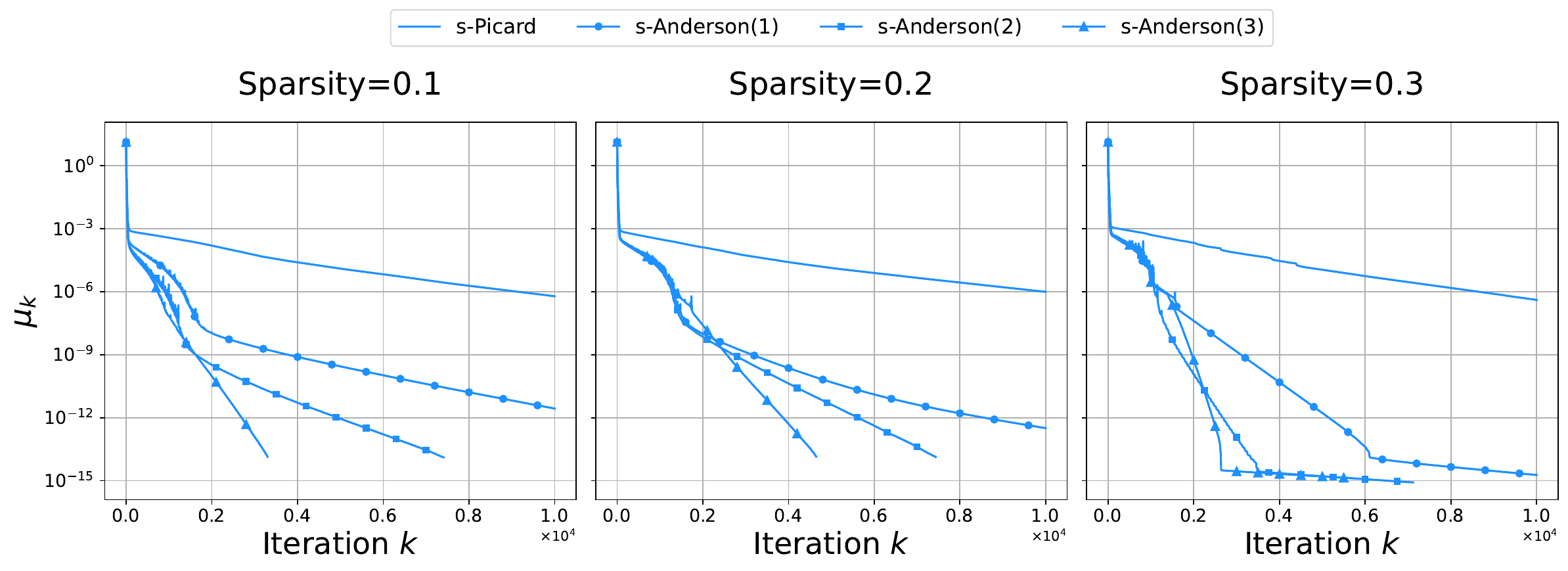}
    \caption{Variation of the smoothing parameters for the ENR problem in different cases}
	\label{fig:enrsmooth}
    \vskip -0.3in
\end{figure}

    \begin{table*}[ht]
    \vskip -0.2in
    \centering
    \caption{Comparisons of the iteration and computational time cost by different algorithms on the ENR problem. Iteration, average per-iteration time (s) and total running time (s) are abbreviated as “$k$",“$t/k$" and “$t$", respectively}
    \label{table:enr}
    \resizebox{\textwidth}{!}{
    \begin{tabular}{l D{/}{/}{5} D{/}{/}{5} >{\columncolor{gray!20}[2pt][0pt]}D{/}{/}{5} D{/}{/}{5} D{/}{/}{5} >{\columncolor{gray!20}[2pt][0pt]}D{/}{/}{5} D{/}{/}{5} D{/}{/}{5} >{\columncolor{gray!20}[2pt][0pt]}D{/}{/}{5}}
        \toprule
        \multirow{1}{*}{{Cost}} & \multicolumn{3}{c}{{Sparsity=0.1}} & \multicolumn{3}{c}{{Sparsity=0.2}} & \multicolumn{3}{c}{{Sparsity=0.3}}\\
        \multicolumn{1}{c}{$({\rm tol} = 10^{-6}/10^{-15})$} & \multicolumn{1}{c}{$k(\times 10^{3})$} & \multicolumn{1}{c}{$t/k(\times 10^{-3})$} & \multicolumn{1}{c}{$t$} & \multicolumn{1}{c}{$k(\times 10^{3})$} & \multicolumn{1}{c}{$t/k(\times 10^{-3})$} & \multicolumn{1}{c}{$t$} & \multicolumn{1}{c}{$k(\times 10^{3})$} & \multicolumn{1}{c}{$t/k(\times 10^{-3})$} & \multicolumn{1}{c}{$t$}\\
        \cmidrule(lr){1-1}\cmidrule(lr){2-4} \cmidrule(lr){5-7}\cmidrule(lr){8-10}
    
     Picard & 10/- & \textbf{0.39}/- & 3.87/- & 9.31/- & \textbf{0.43}/- & 4.04/- & 8.85/- & \textbf{0.46}/- & 4.03/- \\
     Anderson (1) & 5.90/- & 0.60/- & 3.57/- & 5.03/- & 0.71/- & 3.56/- & 4.69/- & 0.74/- & 3.42/-\\
     Anderson(2) & 1.95/8.08 & 0.67/\textbf{0.65} & 1.30/5.26 & 1.78/8.37 & 0.80/\textbf{0.75} & 1.42/6.27 & 1.67/8.73 & 0.89/\textbf{0.75} & 1.48/6.56 \\
     Anderson(3) &1.48/5.71 & 0.70/0.70 & 1.04/3.99 &  1.42/6.60 & 0.77/0.81 & 1.09/5.41 & 1.26/6.43 & 0.79/0.74 & 1.00/\textbf{4.7}4\\
     \hdashline
     s-Picard & 4.97/- & 0.68/- & 3.37/- & 5.04/- & 0.78/- & 3.94/- &  4.90/- & 0.92/- & 4.49/-\\
    s-Anderson(1) & 0.86/- & 0.96/- & 0.82/- & 0.97/- & 1.07/- & 1.03/- & 0.95/7.50 & 1.03/1.04 & \textbf{0.99}/7.81 \\
    s-Anderson(2) & 0.48/7.41 & 1.01/0.99 & 0.48/7.34 & 0.96/7.44 & 1.05/1.06 & \textbf{1.00}/7.94 & 1.07/7.12 & 1.10/1.08 & 1.18/7.70 \\
    s-Anderson(3) & 0.43/3.30 & 1.04/1.04 & \textbf{0.44}/\textbf{3.44} & 0.99/4.65 & 1.07/1.10 & 1.07/\textbf{5.11} & 1.00/5.70 & 1.09/1.08 & 1.10/6.12\\
    \hdashline
    AA1 & 1.07/6.82 & 6.76/6.68 & 7.21/45.57 & \textbf{0.37}/9.48 & 6.71/6.62 & 2.46/65.17 & \textbf{0.35}/- & 7.12/- & 2.48/- \\
    AA1-safe$^1$\footnotetext[1]{The code is available from \url{https://github.com/cvxgrp/nonexp_global_aa1}} &\textbf{0.37}/\textbf{1.05} & 6.94/6.81 & 2.55/7.18 & 0.38/\textbf{0.96} & 6.97/6.84 & 2.67/6.60 & 0.38/\textbf{0.83} & 8.59/8.05 & 3.29/6.67\\
    FISTA & 1.81/- & 3.88/- & 7.02/- & 1.98/- & 2.27/- & 4.50/- & 2.08/-  & 0.88/-  & 1.83/-  \\
    \bottomrule
    \end{tabular}}
    \vskip -0.1in
\end{table*}

To further illustrate the effectiveness of s-Anderson(m) for solving the ENR problem in large-scale and real-world datasets, we conduct an extra comparative experiment with SSNAL \cite{li2018highly} on the Sparco collection \cite{van2007sparco}. We select 10 real-valued sparse reconstruction problems in the Sparco collection. Similar to \cite{li2018highly}, we introduce a 60dB noise to $\b$ and use the following relative KKT residual $\eta$ to measure the accuracy of the approximate solution $\u$: 
\begin{equation*}
    \eta = \frac{\|\u - {\rm prox}_{(\alpha\lambda/2)\|\cdot\|_1}\left(\u - \alpha \left(\A^T(\A\u - \b) + \frac{\lambda}{2}\u\right) \right)\|_{\infty}}{1+\|\u\|_{\infty}+ \max\{\|\A\u - \b\|_{\infty}, \|(\sqrt{2\lambda}/2) \u\|_{\infty}\}}
\end{equation*}
where $\alpha = 1.8/L$ is a fixed stepsize with $L = \|\A\|^2+\lambda/2$. For a given tolerance ${\rm tol}= 5\times 10^{-7}$, we stop the tested algorithms when $\eta < {\rm tol}$ at $\u_k$. We set $k_{\max} = 20000$ for our algorithm and $k_{\max} = 1000$ for SSNAL\footnote[1]{The code is available at \url{https://github.com/MatOpt/SuiteLasso}}, which is the same as in \cite{li2018highly}. Since SSNAL was originally designed to solve the LASSO problem, we need to reformulate the problem \cref{enr} as follows:
\begin{equation}
    \min_{\u \in \mathbb{R}^n}  \frac{1}{2}\left\Vert\left(\begin{array}{c}
    \A\u \\
    \sqrt{2\lambda}/2 \u \\
    \end{array}\right) - \left(\begin{array}{c}
    \b \\
    \boldsymbol{0} \\
    \end{array}\right) \right\Vert^2 + \frac{\lambda}{2} \|\u\|_1.
\end{equation}

In \cref{table:enr2}, we report the relative KKT residual $\eta$ and the running time of s-Anderson(m) and SSNAL for solving the ENR problem on the selected datasets. The experimental results indicate that while s-Anderson(m) and SSNAL achieve comparable accuracy, their efficiency depends on the system type. s-Anderson(m) is notably superior for underdetermined systems ($M<n$), while SSNAL demonstrates greater efficiency in handling determined systems ($M=n$).
\begin{table*}[ht]
    \vskip -0.2in
    \centering
    \caption{Comparisons of the relative KKT residual $\eta$ and the running time (s) with SSNAL on the ENR problem. In the table, s-A(3) = s-Anderson(3), SS = SSNAL}
    \label{table:enr2}
    \begin{tabular}{c c c c c c c}
        \toprule
        \multirow{2}{*}{{Category}} & \multirow{2}{*}{{Probname}} & \multirow{2}{*}{{$(M, n)$}} & \multicolumn{2}{c}{{$\eta$}} & \multicolumn{2}{c}{{time(s)}}\\
        % \cdashline{4-4}

        & &  & s-A(3) & SS & s-A(3) & SS \\
        % & a & b & c & d & e
        \cmidrule(lr){1-1}\cmidrule(lr){2-2}\cmidrule(lr){3-3} \cmidrule(lr){4-5} \cmidrule(lr){6-7}
        \multirow{4}{*}{{$M = n$}}& blocksig & (1024,1024)  & \textbf{1.5e-7} & {4.1e-7} & \textbf{0.02} & 0.05 \\
        & blkheavi & (128,128)  & {4.8e-7} & \textbf{5.1e-9} & 0.23  & \textbf{0.12} \\
        & blknheavi & (1024,1024)  & 4.8e-7 & \textbf{3.1e-8} & 0.62 & \textbf{0.13}\\
        & spiketrn & (1024,1024) & 4.8e-7 & \textbf{4.1e-7} & 3.67 & \textbf{0.17} \\
        \hdashline 
        \multirow{8}{*}{{$M < n$}} & cosspike & (1024,2048) & \textbf{1.9e-8} & 3.8e-8 & \textbf{0.05} & 0.18 \\
        & gausspike & (256,1024) & \textbf{4.8e-7} & 5.2e-7 & \textbf{0.39} & 0.77 \\
        & gcosspike & (300,2048) & \textbf{2.7e-7} & \textbf{2.7e-7} & 0.46 & \textbf{0.37} \\
        & jitter & (200,1000) & \textbf{1.5e-7} & 2.1e-7 & \textbf{0.02} & 0.21\\
        & p3poly & (600,2048) & \textbf{6.8e-8} & 7.5e-8 & 8.94 & \textbf{8.13} \\
        & sgnspike & (600,2560) & \textbf{9.1e-8} & 3.6e-7 & \textbf{0.10} & 0.79 \\
    \bottomrule
    \end{tabular}
    \vskip -0.1in
    \end{table*}

\end{example}

\subsection{{Generalized absolute value equation}}\label{subsec2}
Consider the following generalized absolute value equation (GAVE)
\begin{equation}\label{gave}
    \A\u - \B|\u|=\b,
\end{equation}
where $\A, \B \in \mathbb{R}^{n\times n}$, $\b\in \mathbb{R}^{n}$, and $|\u|$ denotes the component-wise absolute value of the vector $\u$. GAVE (\ref{gave}) degenerates into the absolute value equation (AVE) when $\B = \I$. In 2007, Mangasarian \cite{mangasarian2007absolute} proved that AVE is an NP-hard problem. Then, Prokopyev \cite{prokopyev2009equivalent} proved that testing whether AVE had a unique solution or multiple solutions is also an NP-complete problem. 

We can regard (\ref{gave}) as a fixed point problem as follows:
\begin{equation}\label{gave_fixed}
    \u = G(\u) :=(\I-\A)\u+\B|\u|+\b.
\end{equation}
As for the contraction of $G$, we have
\begin{equation*}
    \begin{aligned}
        \| G(\u) - G(\v)\| = &\|(\I-\A)(\u-\v)+\B(|\u| -|\v|)\| \leq \|(\I-\A)(\u-\v)\| +\|\B(|\u| -|\v|)\| \\
        \leq& (\|\I-\A\|+\|\B\|)\|\u-\v\|,\\
    \end{aligned}
\end{equation*}
so if $\|\I-\A\|+\|\B\| <1$, then $G$ in (\ref{gave_fixed}) is a contraction mapping. In particular, if $\A$ is an invertible matrix, then we can get a simpler form, i.e., $G(\u) =\A^{-1}\B|\u|+\A^{-1}\b$, and $G$ is a contraction mapping if $\|\A^{-1}\B\|<1$.

Similarly, by setting $|\u|= \max\{\u, 0\} + \max\{-\u, 0\}$, we can set the smoothing function of $G$ in \cref{gave_fixed} by $\mathcal{G}(\u,\mu) = (\I-\A)\u + \B(\Phi(\u,\mu)+\Phi(-\u,\mu))+\b$, 
where $\Phi(\v,\mu)=(\phi(v_1,\mu),\ldots, \phi(v_n,\mu))^T$ for $\v=(v_1, \ldots, v_n)^T\in\mathbb{R}^n$ and $\phi$ is defined as in \cref{sm1}.
\begin{example}
    A journal bearing consists of a rotating cylinder which is separated from a surface (the bearing) by a thin film of lubricating fluid. A central issue in lubrication theory is the analysis of this film, including its thickness, pressure distribution, frictional forces, and load-carrying capacity. These factors are critical for ensuring optimal performance and longevity of the bearing system. The objective of this example is to find the distribution function of the pressure $p$ in the lubricant and $h$ represents the thickness function of the lubricant film. We present an infinitely long journal bearing as a numerical experiment and the mathematical model for this problem can be formulated as follows \cite{cottle2009linear,pinkus1962theory}: 
\begin{equation}\label{org_jourbear}
\begin{aligned}
\frac{d}{d t}\left[h^3(t) \frac{d p}{d t}\right] & =\frac{d h(t)}{d t}, & & 0<t<\tau, \\
p(t) & =0, & & \tau \leq t \leq 2, \\
p(0) & =0, & & \\
\frac{d}{d t} p(t) & =0, & & t= \tau, \\
h(t) &= (1+\epsilon \cos \pi t) / \sqrt{ \pi},
\end{aligned}
\end{equation}
where $\tau \in (0, 2)$ is a constant and $\epsilon \in [0,1)$ is the eccentricity ratio. This model assumes the lubricating film is so thin that pressure variations in the direction orthogonal to the journal's axis can be considered negligible.

By finite difference methods, we can discretize the problem \cref{org_jourbear} and subdivide the interval $[0, 2] $ into $n+1$  subintervals,  each of length  $\Delta t$,
so that $\Delta t = 2/ (n+1)$, and denote $h_i = h(i \Delta t)$, $h_{i-1/2}=h([i-\frac{1}{2}]\Delta t), i= 1, \dots, n$. Let $\bm{P} \in \mathbb{R}^{n+2}$ be the numerical solution of \cref{org_jourbear}. 
Cryer \cite{cryer1971method} proved that the discrete approximation of \cref{org_jourbear} is equivalent to a linear complementarity problem, i.e., find a vector $\p\in \mathbb{R}^n$ such that
\begin{equation}\label{jourlcp}
    \p \geq 0, \quad \A\p-\b \geq 0, \quad \p^T(\A\p-\b)=0
\end{equation}
where
$$
\A=\left(\begin{array}{rcll}
(h_{1+1/2})^3+(h_{1-1/2})^3 & -(h_{1+1/2})^3 &  &\\
-(h_{2-1/2})^3 & (h_{2+1/2})^3+(h_{2-1/2})^3 & -(h_{2+1/2})^3&  \\
   \ddots & \qquad \ddots & \qquad \qquad \ddots&  \\
  &    &   & -(h_{(n-1)+1/2})^3\\
   & & -(h_{n-1/2})^3 & (h_{n+1/2})^3+(h_{n-1/2})^3
   \end{array}
   \right)
$$
and 
$$
\b = \left(-(\Delta t)(h_{1+1/2} - h_{1-1/2}), \dots , -(\Delta t)(h_{n+1/2} - h_{n-1/2}) \right)^T.
$$ 
Additionally, Cryer gave the relationship between $\bm{P}$ and $\p$ as $\bm{P} = [0,\p, 0]^T$. Notice that $\A$ is a symmetric irreducibly diagonally dominant matrix with positive diagonal entries, so that $\A$ is positive definite \cite{varga1962iterative}. 

For the linear complementarity problem \cref{jourlcp}, by setting $\p = |\u|+\u$ and $\A\p-\b= |\u| - \u$, it can be viewed as a fixed point problem as follows \cite{dong2009modified}:
$$\u = G(\u) = (\I+\A)^{-1}(\I-\A)|\u|+(\I+\A)^{-1} \b.$$ 
Also, the smoothing function $\mathcal{G}$ of $G$ can be defined as
\begin{equation}\label{smoothjourlcp}
    \mathcal{G}(\u, \mu) = (\I+\A)^{-1}(\I-\A)(\Phi(\u, \mu)+\Phi(-\u, \mu)) + (\I+\A)^{-1} \b,
\end{equation}
where $\Phi(\v,\mu)$ is defined as before. Since $\A \in \mathbb{R}^{n\times n}$ is a symmetric positive definite matrix, $G$ is contractive, and the fixed point problem has a unique solution, which can be transformed to the solution of \cref{jourlcp} by setting $\p = |\u|+\u$. 

In this problem, we set ${\rm tol }=10^{-12}$, $k_{\max} = 20000$, $\epsilon=0.4$ and choose the initial point $\u_0 = {\rm 15 * np.random.\newline randn(n, 1)}$. We consider problem \cref{jourlcp} with $n=500$. The first and second subplots in Fig. \ref{fig:jourfull} show the convergence of $\|F_k\|/\|F_0\|$ by Anderson(m) and s-Anderson(m) versus the iteration number and running time, respectively. The last one indicates that the smoothing parameter is tending towards zero as $k\rightarrow \infty$. Fig. \ref{fig:jourtime} demonstrates the convergence performance of Anderson(3) (with various fixed smoothing parameters) for finding the fixed point of $\mathcal{G}(\cdot, \mu)$, compared to our algorithm (with an updated smoothing parameter) for finding the fixed point of ${G}$.
Additionally, the numbers of iterations required by Anderson(m) and s-Anderson(m) to find an iterate $\u_k$ satisfying \cref{stop} are shown in \cref{table:bear}.

\begin{figure}[h]
    \centering
    \includegraphics[width=1\textwidth]{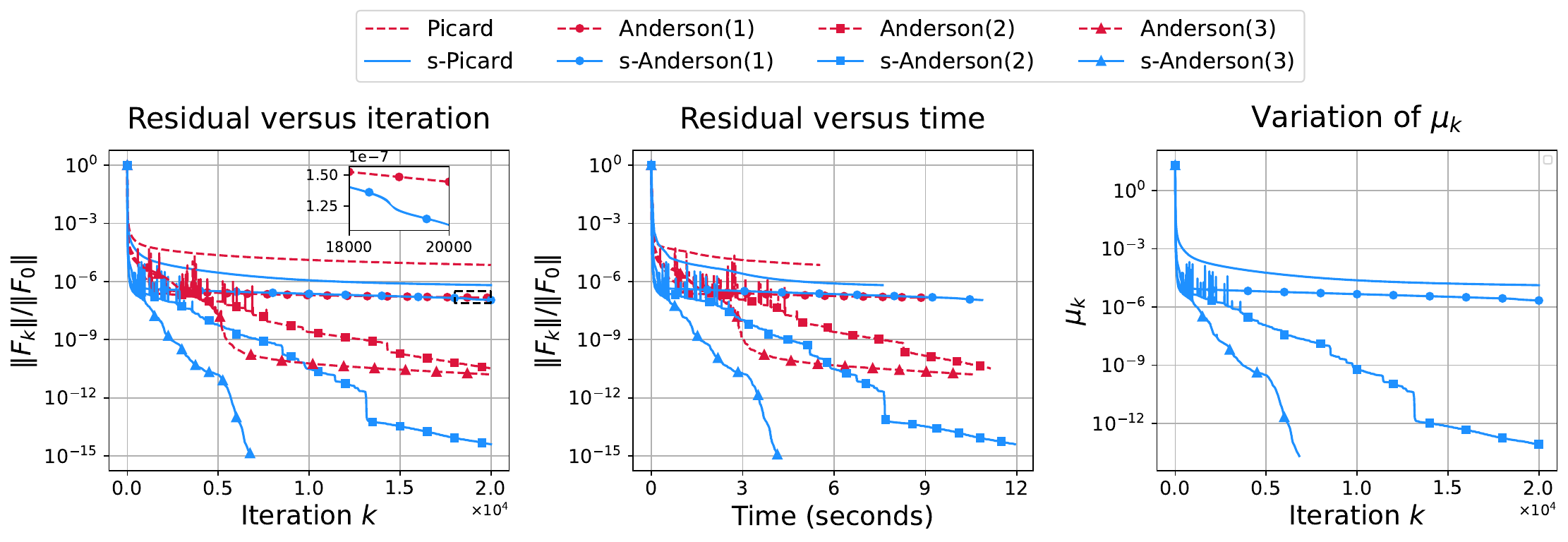}
    \caption{Free boundary problems for infinite journal  bearings with $n=500$}
    \label{fig:jourfull}
    \vskip -0.2in
\end{figure}

\begin{figure}[h]
    \includegraphics[width=1\textwidth]{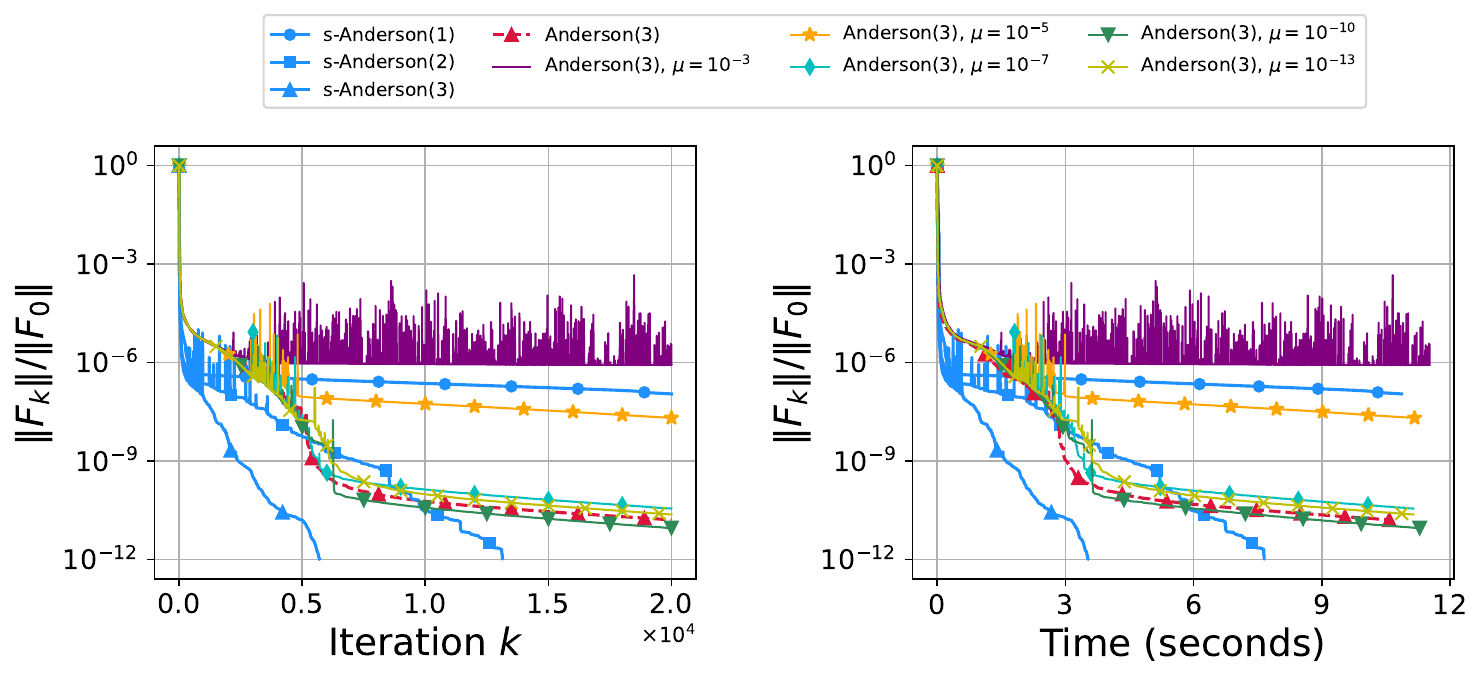}
    \caption{Comparisons of some different algorithms in iterations and {\rm CPU time} for \cref{jourlcp} with $n=500$}
    \label{fig:jourtime}
    \vskip -0.3in
\end{figure}

Based on these numerical results, we gain the following observations:
\begin{itemize}
    \item Within a certain range, the local convergence of $\|F_k\|/\|F_0\|$ by s-Anderson($m$) becomes faster as $m$
    increases (see also \cref{table:bear} for additional tests with different values of $m$). However, for a large value of $m$, the method may fail to find the solution in some instances.
    \item With different values of $n$, s-Anderson(m) exhibits superior numerical performance than Anderson(m), both in terms of the number of iterations and total running time.
    \item When $\mu < 10^{-3}$, Anderson(m) to \cref{smoothjourlcp} with a fixed smoothing parameter can find a proper approximate fixed point of $G$ and $\mu =10^{-10}$ is the best fixed smoothing parameter for Anderson(3) to \cref{smoothjourlcp}. However, s-Anderson(m) with adaptive smoothing parameter is also more effective and faster.
\end{itemize}

\begin{remark}
    The complementarity conditions can also be reformulated as a class of fixed point problems, and it can be applied to the KKT system formulations. Consider the constrained optimization problem 
    \begin{equation}\label{opt}
        \min \quad f(\x) \quad
        {\rm s.t.} \quad g(\x)\geq 0,
    \end{equation}
    where $f: \mathbb{R}^n \rightarrow\mathbb{R}$ and $g:\mathbb{R}^n \rightarrow\mathbb{R}^m$ are Lipschitz continuously differentiable. Then $(\x, \boldsymbol{\lambda})$ is a KKT point of \eqref{opt} if it satisfies $\nabla_{\x} L(\x,\boldsymbol{\lambda})=0, \boldsymbol{\lambda}\geq0, g(\x)\geq0$ and $\boldsymbol{\lambda}^\top g(\x) = 0$,
    where $L(\x,\boldsymbol{\lambda})=f(\x)-{\boldsymbol{\lambda}}^\top g(\x)$. Then the KKT system is equivalent to the following fixed point problem
        \begin{equation*}
        \left(\begin{array}{c}
        \x  \\
        \boldsymbol{\lambda}  \\
        \end{array}\right)=\left(\begin{array}{c}
        \x - \gamma \nabla_{\x}L(\x,\boldsymbol{\lambda})  \\
        \max\{\boldsymbol{\lambda}-g(\x),0\}  \\
        \end{array}\right) 
        \end{equation*}
    with $\gamma\in(0,1)$. Under contractive conditions, we can also construct a smoothing function of it based on the proposed smoothing function of $\max$ in \eqref{sm1}. 
\end{remark}
\end{example}

\subsection{{Projected gradient descent with non-negative constraints}}
Consider the following  minimization problem with non-negative constraints
\begin{equation}\label{pga}
    \begin{aligned}
        \min \limits_{\u \in \Omega} \, f(\u), 
    \end{aligned}
\end{equation}
where $\Omega := \{\u\in \mathbb{R}^n: \u \geq 0\}$ and $f: \mathbb{R}^n \rightarrow \mathbb{R}$ is $L$-Lipschitz smooth and convex on $\Omega$. We can use the projected gradient descent algorithm to solve (\ref{pga}), i.e.,
\begin{equation}\label{pgda}
    \u_{k+1} = P_{\Omega}(\u_k - \alpha \nabla f(\u_k)),
\end{equation}
where $\alpha \in (0, 2/L)$ is a suitable stepsize. 

Note that $P_{\Omega}(\u) = \max \{\u, 0\}$, by setting $Q:\mathbb{R}^n\rightarrow \mathbb{R}^{n}$ with
$Q(\u) = \u - \alpha \nabla f(\u)$, then the problem (\ref{pga}) can be expressed by the fixed point problem as follows: 
\begin{equation}\label{pga-fixed}
    \u = G(\u) := \max \{Q(\u), 0\}.
\end{equation}
Using $\phi$ in \cref{sm1}, we can set the smoothing function by $\mathcal{G}(\u, \mu) = \Phi(Q(\u),\mu)$, where $\Phi(\v,\mu)=(\phi(v_1,\mu),\ldots,\phi(v_n,\mu))^T$ for $\v=(v_1, \ldots, v_n)^T\in\mathbb{R}^n$.

Similar to the discussion in \cref{subsec1}, $G$ in (\ref{pga-fixed}) is a contraction mapping under \cref{str_mon} when $\alpha \in (0, \min\{2/L, 2 \tau/L^2\})$.

\begin{table*}[ht]
        \vskip -0.2in
        \centering
        \caption{Numerical results of Anderson(m) and s-Anderson(m) for different cases}
        \label{table:bear}
        \begin{tabular}{c  D{/}{/}{4} D{/}{/}{6} D{/}{/}{6} D{/}{/}{6} D{/}{/}{6} }
        \toprule
         \multirow{1}{*}{{Method}} & \multicolumn{5}{c}{{Anderson(m)/s-Anderson(m)}}\\
         \cmidrule(lr){2-6} 
         Parameters & \multicolumn{1}{c}{$m=1$} & \multicolumn{1}{c}{$m=2$} &  \multicolumn{1}{c}{$m=3$} &  \multicolumn{1}{c}{$m=5$} &  \multicolumn{1}{c}{$m=9$}\\
         \midrule
         $n=100$  & -/8517 & 8260/\underline{3458} & 4551/\underline{2313} & 1564/\underline{1102} & 1122/\textbf{\underline{860}}\\
         $n=200$ & -/-  & 19330/\underline{6875} & 6131/\underline{2927}& 2524/\underline{\textbf{2465}} & 4387/\underline{7209} \\
         $n=300$  & -/- & 19589/\underline{8693} & 7611/\underline{4928} & 4821/\underline{\textbf{3444}} & \underline{8775}/- \\
         $n=500$  & -/- & -/\underline{13139} & -/\underline{\textbf{5714}} & \underline{9486}/- & \underline{18485}/-  \\
        \bottomrule
        \end{tabular}
        \vskip -0.2in
\end{table*}

\begin{example}
    Consider the following non-negative regularized logistic regression problem:
\begin{equation}\label{nls}
    \min \limits_{\u \geq 0} \, {\displaystyle  \frac{1}{ M}\sum _{i=1}^{M} \log(1+\exp(-b_i \u^T \boldsymbol{a_i}))+\frac{\lambda}{2}\|\u\|^2},
\end{equation}
where $b_i = \pm 1$ are the labels, and $\boldsymbol{a_i} \in \mathbb{R}^n$ are the feature vectors.  
The regularization parameter $\lambda$ is chosen as $\lambda=0.01$. In this experiment, we select two well-known datasets in the field of machine learning, in which the coefficient matrices corresponding to these datasets are ill-conditioned, namely Madelon\footnote[1]{The Madelon dataset is available from \url{https://archive.ics.uci.edu/datasets}} ($\kappa \approx 2.1 \times 10^4$) and MARTI0 \footnote[2]{The MARTI0 dataset is available from \url{https://www.causality.inf.ethz.ch/home.php}} ($\kappa \approx 7.5 \times 10^3$), where $\kappa$ denotes the condition number of the corresponding coefficient matrix. We use the projected gradient descent algorithm to solve \cref{nls}, which can be viewed as the Picard iteration to find the following fixed point problem: 
\begin{equation}\label{nls-fixed}
    \u = G(\u) := \max \left \{\u - \alpha \left(\frac{1}{M}\sum_{i=1}^{M}-\frac{b_i \boldsymbol{a_i}}{1+ \exp(b_i \u^T \boldsymbol{a_i})} +\lambda \u \right), 0 \right\}.
\end{equation}
We choose stepsize $\alpha = 1.99/(L+\lambda)$, where $L = \|\A\|^2/4M$ is an upper bound of the largest Hessian eigenvalue for the logistic loss \cite{scieur2016regularized} with $\A = [\boldsymbol{a}_1, \cdots, \boldsymbol{a}_M]$, so $G$ in \cref{nls-fixed} is a contractive mapping.

Since the coefficient matrix in this problem may be ill-conditioned, we introduce a regularization term in (\ref{s-anderson-m}) to improve the stability of our algorithm. We set $\u_0 = {\rm 20 * np.random.randn(n, 1)}$, ${\rm tol }=10^{-15}$, $k_{\max} = 10000$ for Madelon and $\u_0 ={\rm 2e4 * np.random.randn(n, 1)}$, ${\rm tol }=10^{-15}$, $k_{\max} = 10000$ for MARTI0, respectively. The convergence performance of $\|F_k\|/\|F_0\|$ is illustrated in Fig. \ref{fig:nls}. For different values of $m$, the convergence rates, defined as $(\|F_k\|/\|F_0\|)^{1/k}$, are listed in \cref{table:nnls}, which indicates that s-Anderson(m) outperforms Anderson(m) in almost all cases and s-Anderson(3) achieves the best results. In summary, these results demonstrate the effectiveness of s-Anderson(m) in addressing this class of ill-conditioned problems.

\begin{table*}[ht]
        \vskip -0.2in
        \centering
        \caption{Values of $(\|F_k\|/\|F_0\|)^{1/k}$ by Anderson(m) and s-Anderson(m) for different datasets}
        \label{table:nnls}
        \resizebox{\textwidth}{!}{
        \begin{tabular}{l c c c c c}
        \toprule
         \multirow{2}{*}{{Dataset}} & \multirow{2}{*}{Criterion} & \multicolumn{4}{c}{{Anderson(m)/s-Anderson(m)}} \\
         \cmidrule(lr){3-6} 
          & & $m=0$ & $m=1$ & $m=2$ & $m=3$\\
        \midrule
         \multirow{2}{*}{Madelon} & r-factor & {9.986e-01}/{9.986e-01} & 9.973e-01/\underline{9.954e-01}& 9.972e-01/\underline{9.922e-01}&9.971e-01/\underline{\textbf{9.819e-01}} \\
         & $k$ & 10000/10000 & 10000/\underline{7520} & 10000/\underline{4440}  & 10000/\underline{\textbf{1901}} \\

          \hdashline 
         \multirow{2}{*}{MARTI0} & r-factor & 9.988e-01/9.988e-01 & 9.973e-01/\underline{9.965e-01}& 9.972e-01/\underline{9.936e-01}& 9.972e-01/\underline{\textbf{9.675e-01}}\\
         & $k$ & 10000/10000 & 10000/\underline{9086} & 10000/\underline{5642}  & 10000/\underline{\textbf{1045}} \\
        \bottomrule
        \end{tabular}
        }
        \vskip -0.1in
\end{table*}
\begin{figure}[ht]
    \vskip -0.1in
	\centering
	\subfigure[Madelon]{
		\begin{minipage}[b]{0.455\textwidth}
			\includegraphics[width=1\textwidth]{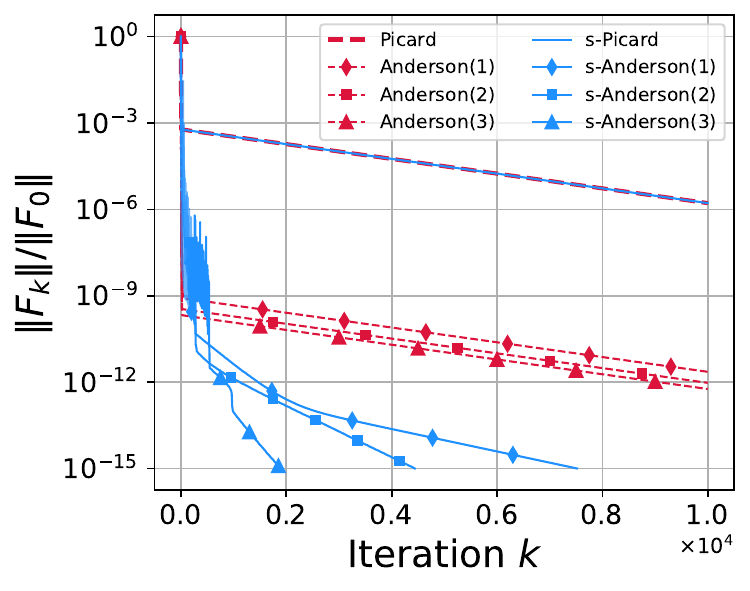}
		\end{minipage}
		\label{fig:madelon}
	}
    	\subfigure[{MARTI0}]{
    		\begin{minipage}[b]{0.455\textwidth}
   		 	\includegraphics[width=1\textwidth]{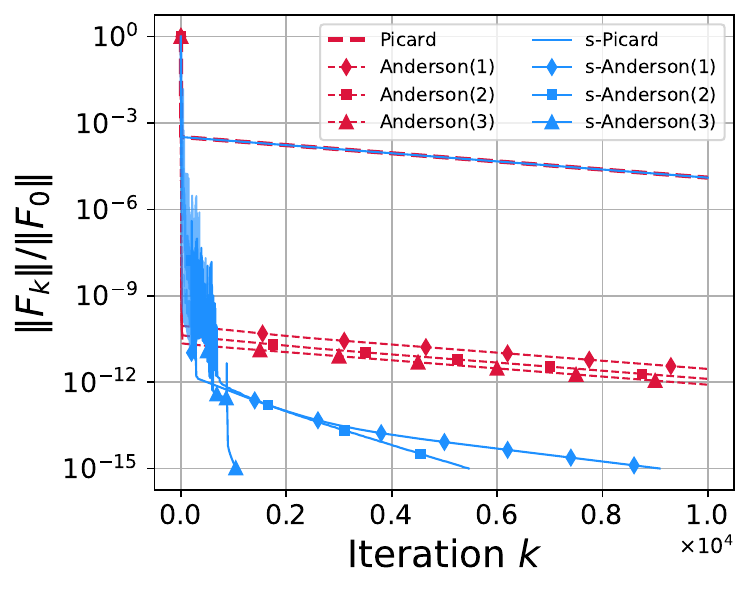}
    		\end{minipage}
		\label{fig:marti0}
    	}
        \vskip -0.15in
	\caption{Convergence performance of Anderson(m) and s-Anderson(m) for problem \eqref{nls}}
	\label{fig:nls}
    \vskip -0.4in
\end{figure}
\end{example}

\section{Conclusions} 
In this paper, we propose a novel algorithm (named Smoothing Anderson(m)) that combines the Anderson acceleration method with a smoothing approximation of $G$ for solving the nonsmooth contractive fixed point problem. Different from most existing studies of Anderson acceleration method, we propose a novel Smoothing Anderson(m) algorithm, in which the smoothing parameter is updated adaptively according to the newest $m_k+1$ residual function values. Based on this setting, we prove the r-linear convergence of the proposed Smoothing Anderson(m) algorithm in \cref{theorem3}, and its q-linear convergence when $m=1$ in \cref{theorem4-s}. Some numerical experiments are shown to demonstrate the theoretical results of this paper and indicate that our proposed algorithm outperforms most accelerated algorithms in terms of the number of iterations and the running time of CPU. 
% To the best of our knowledge, most Anderson acceleration algorithms are proved to be locally linearly convergent or with global convergence but without linear convergence guarantee in theory, but none of them has both properties {\color{red}without assuming differentiability at the fixed point}, which gives an interesting topic for the further research.
To the best of our knowledge, most Anderson acceleration algorithms are proved either to be locally linearly convergent or to converge globally without a linear convergence guarantee in theory, but none of them has both properties, which gives an interesting topic for further research.

\section*{Acknowledgments}
The authors are also grateful to the associate editor Prof. Ulbrich and the reviewers for their many valuable comments and constructive suggestions. Their careful reading and thoughtful feedback greatly enhanced the clarity, accuracy, and overall presentation of this work.

\vspace{-0.1in}
\bibliographystyle{siamplain}
\bibliography{references}

@book{Ortega,
  title={Iterative Solution of Nonlinear Equations in Several Variables},
  author={Ortega, James M and Rheinboldt, Werner C},
  year={2000},
  address = {Philadelphia},
  publisher={SIAM}
}

@article{Anderson,
  title={Iterative procedures for nonlinear integral equations},
  author={Anderson, Donald G},
  journal={J. ACM},
  volume={12},
  number={4},
  pages={547--560},
  year={1965},
  publisher={ACM New York, NY, USA}
}

@article{BCK,
  title={Anderson acceleration for a class of nonsmooth fixed-point problems},
  author={Bian, Wei and Chen, Xiaojun and Kelley, Carl T},
  journal={SIAM J. Sci. Comput.},
  volume={43},
  number={5},
  pages={S1--S20},
  year={2021},
  publisher={SIAM}
}

@article{ChenKelley2015,
  title={Convergence of the {EDIIS} algorithm for nonlinear equations},
  author={Chen, Xiaojun and Kelley, Carl T},
  journal={SIAM J. Sci. Comput.},
  volume={41},
  number={1},
  pages={A365--A379},
  year={2019},
  publisher={SIAM}
}

@article{TothKelley2015,
  title={Convergence analysis for {Anderson} acceleration},
  author={Toth, Alex and Kelley, Carl T},
  journal={SIAM J. Numer. Anal.},
  volume={53},
  number={2},
  pages={805--819},
  year={2015},
  publisher={SIAM}
}

@article{Walker,
  title={Anderson acceleration for fixed-point iterations},
  author={Walker, Homer F and Ni, Peng},
  journal={SIAM J. Numer. Anal.},
  volume={49},
  number={4},
  pages={1715--1735},
  year={2011},
  publisher={SIAM}
}

@article{zhang2020globally,
  title={Globally convergent type-{I} {Anderson} acceleration for nonsmooth fixed-point iterations},
  author={Zhang, Junzi and O'Donoghue, Brendan and Boyd, Stephen},
  journal={SIAM J. Optim.},
  volume={30},
  number={4},
  pages={3170--3197},
  year={2020},
  publisher={SIAM}
}

@inproceedings{wei2021class,
  title={A class of short-term recurrence {Anderson} mixing methods and their applications},
  author={Wei, Fuchao and Bao, Chenglong and Liu, Yang},
  booktitle={Proceedings of the International Conference on Learning Representations},
  url = {https://openreview.net/forum?id=_X90SIKbHa},
  year={2022}
}

@article{peng2018anderson,
  title={Anderson acceleration for geometry optimization and physics simulation},
  author={Peng, Yue and Deng, Bailin and Zhang, Juyong and Geng, Fanyu and Qin, Wenjie and Liu, Ligang},
  journal={ACM Trans. Graph.},
  volume={37},
  number={4},
  pages={1--14},
  year={2018},
  publisher={ACM New York, NY, USA}
}

@article{mangasarian2007absolute,
  title={Absolute value programming},
  author={{Mangasarian}, Olvi Leon},
  journal={Comput. Optim. Appl.},
  volume={36},
  pages={43--53},
  year={2007},
  publisher={Springer}
}

@article{evans2020proof,
  title={A proof that {Anderson} acceleration improves the convergence rate in linearly converging fixed-point methods (but not in those converging quadratically)},
  author={Evans, Claire and Pollock, Sara and Rebholz, Leo G and Xiao, Mengying},
  journal={SIAM J. Numer. Anal.},
  volume={58},
  number={1},
  pages={788--810},
  year={2020},
  publisher={SIAM}
}

@article{pollock2021anderson,
  title={Anderson acceleration for contractive and noncontractive operators},
  author={Pollock, Sara and Rebholz, Leo G},
  journal={IMA J. Numer. Anal.},
  volume={41},
  number={4},
  pages={2841--2872},
  year={2021},
  publisher={Oxford University Press}
}

@article{prokopyev2009equivalent,
  title={On equivalent reformulations for absolute value equations},
  author={Prokopyev, Oleg},
  journal={Comput. Optim. Appl.},
  volume={44},
  pages={363--372},
  year={2009},
  publisher={Springer}
}

@article{zou2005regularization,
  title={Regularization and variable selection via the elastic net},
  author={Zou, Hui and Hastie, Trevor},
  journal={J. R. Stat. Soc. B},
  volume={67},
  number={2},
  pages={301--320},
  year={2005},
  publisher={Oxford University Press}
}

@article{o2016conic,
  title={Conic optimization via operator splitting and homogeneous self-dual embedding},
  author={O'Donoghue, Brendan and Chu, Eric and Parikh, Neal and Boyd, Stephen},
  journal={J. Optim. Theory Appl.},
  volume={169},
  pages={1042--1068},
  year={2016},
  publisher={Springer}
}

@article{cryer1971method,
  title={The method of {Christopherson} for solving free boundary problems for infinite journal bearings by means of finite differences},
  author={Cryer, Colin W},
  journal={Math. Comput.},
  volume={25},
  number={115},
  pages={435--443},
  year={1971}
}

@book{varga1962iterative,
  title={Matrix Iterative Analysis},
  author={Varga, Richard S},
  year={1962},
  address={Englewood Cliffs, NJ},
  publisher={Prentice-Hall}
}

@article{bian2022anderson,
  title={Anderson acceleration for nonsmooth fixed point problems},
  author={Bian, Wei and Chen, Xiaojun},
  journal={SIAM J. Numer. Anal.},
  volume={60},
  number={5},
  pages={2565--2591},
  year={2022},
  publisher={SIAM}
}

@article{beck2009fast,
  title={A fast iterative shrinkage-thresholding algorithm for linear inverse problems},
  author={Beck, Amir and Teboulle, Marc},
  journal={SIAM J. Imaging Sci.},
  volume={2},
  number={1},
  pages={183--202},
  year={2009},
  publisher={SIAM}
}

@book{cottle2009linear,
  title={The linear complementarity problem},
  author={Cottle, Richard W and Pang, Jong-Shi and Stone, Richard E},
  year={2009},
  address = {Philadelphia},
  publisher={SIAM}
 
}

@article{dong2009modified,
  title={A modified modulus method for symmetric positive-definite linear complementarity problems},
  author={Dong, Jun-Liang and Jiang, Mei-Qun},
  journal={Numer. Linear Algebra Appl.},
  volume={16},
  number={2},
  pages={129--143},
  year={2009},
  publisher={Wiley Online Library}
}

@book{pinkus1962theory,
  title={Theory of Hydrodynamic Lubrication},
  author={Pinkus, Oscar and Sternlicht, Beno},
  year={1961},
  address = {New York},
  publisher={McGraw-Hill}
}

@article{ouyang2024descent,
  title={Descent Properties of an {Anderson} Accelerated Gradient Method with Restarting},
  author={Ouyang, Wenqing and Liu, Yang and Milzarek, Andre},
  journal={SIAM J. Optim.},
  volume={34},
  number={1},
  pages={336--365},
  year={2024},
  publisher={SIAM}
}

@article{chen2024non,
  title={Non-stationary {Anderson} acceleration with optimized damping},
  author={Chen, Kewang and Vuik, Cornelis},
  journal={J. Comput. Appl. Math.},
  pages={116077},
  year={2024},
  volume={451},
  publisher={Elsevier}
}

@article{ouyang2023nonmonotone,
  title={Nonmonotone globalization for {Anderson} acceleration via adaptive regularization},
  author={Ouyang, Wenqing and Tao, Jiong and Milzarek, Andre and Deng, Bailin},
  journal={J. Sci. Comput.},
  volume={96},
  number={1},
  pages={5},
  year={2023},
  publisher={Springer}
}

@article{kudin2002black,
  title={{A} black-box self-consistent field convergence algorithm: One step closer},
  author={Kudin, Konstantin N and Scuseria, Gustavo E and Canc{\`e}s, Eric},
  journal={ J. Chem. Phys.},
  volume={116},
  number={19},
  pages={8255--8261},
  year={2002},
  publisher={ American Institute of Physics}
}

@article{garza2012comparison,
  title={Comparison of self-consistent field convergence acceleration techniques},
  author={Garza, Alejandro J and Scuseria, Gustavo E},
  journal={J. Chem. Phys.},
  volume={137},
  number={5},
  eid ={{art. 054110}},
  year={2012},
  publisher={AIP Publishing}
}

@article{li2018highly,
  title={A highly efficient semismooth {Newton} augmented {Lagrangian} method for solving Lasso problems},
  author={Li, Xudong and Sun, Defeng and Toh, Kim-Chuan},
  journal={SIAM J. Optim.},
  volume={28},
  number={1},
  pages={433--458},
  year={2018},
  publisher={SIAM}
}

@article{van2007sparco,
  title={Sparco: A testing framework for sparse reconstruction},
  author={van den Berg, Ewout and Friedlander, M P and Hennenfent, G and Herrmann, F and Saab, R and Y{\i}lmaz, {\"O}},
  journal={ACM Trans. Math. Softw.},
  volume={35},
  number={1},
  pages={1--16},
  year={2009}
}

@article{fu2020anderson,
  title={Anderson Accelerated {Douglas--Rachford} Splitting},
  author={Fu, Anqi and Zhang, Junzi and Boyd, Stephen},
  journal={SIAM J. Sci. Comput.},
  volume={42},
  number={6},
  pages={A3560--A3583},
  year={2020},
  publisher={SIAM}
}

@inproceedings{mai2020anderson,
 title = 	 {Anderson Acceleration of Proximal Gradient Methods},
  author =       {Mai, Vien and Johansson, Mikael},
  booktitle = 	 {Proceedings of the 37th International Conference on Machine Learning},
  pages = 	 {6620--6629},
  year = 	 {2020},
  editor = 	 {{{H. Daum{\'e} III}} and Aarti Singh},
  volume = 	 {119},
  series = 	 {Proceedings of Machine Learning Research},
  month = 	 {13--18 Jul},
  publisher =    {PMLR},
}

@inproceedings{scieur2016regularized,
  title={Regularized nonlinear acceleration},
  author={Scieur, Damien and d'Aspremont, Alexandre and Bach, Francis},
  booktitle={Advances in Neural Information Processing Systems},
  pages={712-720},
  year={2016}
}

@article{daubechies2004iterative,
  title={An iterative thresholding algorithm for linear inverse problems with a sparsity constraint},
  author={Daubechies, Ingrid and Defrise, Michel and De Mol, Christine},
  journal={Comm. Pure Appl. Math.},
  volume={57},
  number={11},
  pages={1413--1457},
  year={2004},
  publisher={Wiley Online Library}
}

@article{tibshirani1996regression,
  title={Regression shrinkage and selection via the lasso},
  author={Tibshirani, Robert},
  journal={J. Roy. Statist. Soc. Ser. B},
  volume={58},
  number={1},
  pages={267--288},
  year={1996},
  publisher={Oxford University Press}
}

@inproceedings{bai2019deep,
  title={Deep equilibrium models},
  author={Bai, Shaojie and Kolter, J Zico and Koltun, Vladlen},
  booktitle={Advances in Neural Information Processing Systems},
  volume={32},
  year={2019}
}

@article{more1996global,
  title={Global methods for nonlinear complementarity problems},
  author={Mor{\'e}, Jorge J},
  journal={Math. Oper. Res.},
  volume={21},
  number={3},
  pages={589--614},
  year={1996},
  publisher={INFORMS}
}

@book{rockafellar1998variational,
  title={Variational analysis},
  author={Rockafellar, R Tyrrell and Wets, Roger JB},
  year={1998},
  address = {Berlin},
  publisher={Springer}
}

@book{ulbrich2011semismooth,
  title={Semismooth Newton methods for variational inequalities and constrained optimization problems in function spaces},
  author={Ulbrich, Michael},
  year={2011},
  address = {Philadelphia},
  publisher={SIAM}
}

@Misc{amsmath,
  author =	 {{American Mathematical Society}},
  title =	 {User's Guide for the \texttt{amsmath} Package
                  (Version 2.0)},
  url =		 {ftp://ftp.ams.org/pub/tex/doc/amsmath/amsldoc.pdf},
  urldate =	 {2015-07-30},
  year =	 2002}

\end{document}